\documentclass[11pt,a4paper]{article}
\usepackage{authblk}
\usepackage[utf8]{inputenc}
\usepackage[english]{babel}
\usepackage{amsthm}
\usepackage{amsmath}
\usepackage{amsfonts}
\usepackage{amssymb}
\usepackage{makeidx}
\usepackage{graphicx}
\usepackage[left=2cm,right=2cm,top=2cm,bottom=2cm]{geometry}
\usepackage{hyperref}
\usepackage{todonotes}
\usepackage{cite}

\newcommand{\aff}{\mathrm{aff}\,}
\def\argmin{ \mathop{{\rm argmin}}}

\newcommand{\co}{\mathrm{conv}\,}
\newcommand{\cone}{\mathrm{cone}\,}
\newcommand{\cl}{\mathrm{cl}\;}
\newcommand{\clco}{\mathrm{\overline{conv}}\,}

\newcommand{\dom}{\mathrm{dom}\,}
\newcommand{\epi}{\mathrm{epi}\,}

\newcommand{\gph}{\mathrm{gph}\,}
\newcommand{\inter}{\mathrm{int}\,}
\newcommand{\para}{\mathrm{par}\,}

\newcommand{\ri}{\mathrm{ri}\,}

\newcommand{\rge}{\mathrm{rge}\,}

\newcommand{\hzn}{\mathrm{hzn}\,}
\newcommand{\lin}{\mathrm{span}\,}

\newcommand{\R}{\mathbb{R}}

\newcommand{\bS}{\mathbb{S}}

\newcommand{\rbar}{\overline{\mathbb R}}

\newcommand{\rp}{\mathbb R\cup\{+\infty\}}
\newcommand{\bR}{\mathbb{R}}

\def\bS{\mathbb{S}}

\newcommand{\tr}{\mathrm{tr}\,}

\newcommand{\id}{\mathrm{id}}

\newcommand{\bN}{\mathbb{N}}
\newcommand{\bE}{\mathbb{E}}

\newcommand{\ip}[2]{\left\langle #1,\, #2\right\rangle}

\newcommand{\set}[2]{\left\{#1\,\left\vert\; #2\right.\right\}}

\newcommand{\cD}{\mathcal{D}}

\newcommand{\cN}{\mathcal{N}}

\newcommand{\AND}{\ \mbox{ and }\ }

\newcommand{\infconv}{\mathbin{\mbox{\small$\square$}}}

\newcommand{\Epi}[2]{{#1}\mbox{-}\mathrm{epi}{\,#2}}

\newcommand{\Kp}{\mbox{-}K^\circ}

\newtheorem{proposition}{Proposition}
\newtheorem{theorem}[proposition]{Theorem}
\newtheorem{corollary}[proposition]{Corollary}
\newtheorem{lemma}[proposition]{Lemma}
\newtheorem{definition}[proposition]{Definition}

\newtheorem{example}[proposition]{Example}
\newtheorem{remark}[proposition]{Remark}

\title{A note on the   $K$-epigraph}

\begin{document}

\author[1]{Armand Gissler}
\author[2]{Tim Hoheisel}

\affil[1]{Department of Mathematics, \'Ecole Normale Sup\'erieure Paris-Saclay}
\affil[2]{Department of Mathematics and Statistics, McGill University, Montr\'eal
}

\maketitle


\abstract{\noindent 
We study the question as to when the closed convex hull of a $K$-convex map equals its $K$-epigraph. In particular, we shed light onto the smallest cone $K$ such that  a given map has convex and closed $K$-epigraph,  respectively. We apply our findings to several examples in matrix space  as well as to convex composite functions.}

\bigskip

\bigskip

\noindent 
{\bf MSC2000 Subject Classification:} 52A41, 65K10, 90C25, 90C46

\section{Introduction}\label{sec:Intro}

\paragraph{Motivation} 
In  a recent paper,  Burke et al. \cite[Corollary 9]{BGH 17}  show that the closed convex hull  of the set $\cD:=\set{(X,\frac{1}{2}XX^T)}{X\in \R^{n\times m}}$ is given by 
\begin{equation*}\label{eq:Mot}
\clco\cD=\set{(X,Y)\in \R^{n\times m}\times \bS^n}{Y\succeq \frac{1}{2}XX^T}.
\end{equation*}
Here, `$\succeq$' is the {\em L\"owner  partial ordering} \cite{HoJ 91} on the symmetric matrices $\bS^n$ induced by the positive semidefinite cone $\bS^n_+$ via    `$A\succeq B$ if and only if $A-B\in \bS^n_+$'. 
 At second glance, the set $\cD\in \R^{n\times m}\times \bS^n$ is simply the graph of the matrix-valued map $F:X\in \R^{n\times m}\mapsto \frac{1}{2}XX^T\in \bS^n$;  and $\clco \cD$ in \eqref{eq:Mot} then  appears to be  a `generalized epigraph'  of $F$ where  the partial ordering on the image space $\bS^n$ (induced by $\bS^n_+$) plays the role of the ordinary ordering of $\R$ (induced by $\R_+$) for scalar-valued functions. 

More generally, given a map $F:\bE_1\to\bE_2$  between two (Euclidean) spaces $\bE_1$ and $\bE_2$ and a cone $K\subset \bE_2$, we  can order $\bE_2$ via `$y\geq_K z$ if and only if $y-z\in K$'.  In view of the above identity, the natural question that arises is the following: {\em When is}
\begin{equation}\label{eq:Q} 
\clco (\gph F)=\set{(x,y)}{y\geq_K F(x) }
\end{equation}
{\em valid?}  Clearly, this can only hold if the set on the right, which will later be called the  {\em $K$-epigraph of $F$}, is itself closed and convex, in which case we  say that  $F$ is  {\em $K$-closed} and {\em $K$-convex},  respectively,  or {\em closed} and {\em convex, with respect  to (w.r.t.) $K$}, respectively.

\paragraph{Related work} The study of $K$-convexity has a long tradition in convex analysis and is  now part of many textbooks, e.g. \cite{BoV 04, RoW 98}:  Borwein \cite{Bor 74} 
pursued an ambitious program of extending
most of convex analysis to cone convex functions including conjugacy, subdifferential analysis,
and duality, laying out much of the groundwork.  Kusraev and Kutateladze  \cite{KK95} take this idea to an even more general setting by considering {\em convex operators} with values in arbitrary ordered vector spaces.  
Pennanen \cite{Pen 99}  develops  a deep  theory of generalized differentiation  for  {\em graph-convex mappings} (these are   called  {\em convex correspondences}  in \cite{KK95}) which contains some  results on $K$-convexity, highly relevant to our study. One of the most important features of  a $K$-convex map $F$ is the fact that the composition $g\circ F$  with a convex function $g$, which is increasing with respect to the ordering induced by $K$, is convex; a fact that has been well observed and utilized widely in the literature \cite{Bot 10,Bot et al. 06,BGW 07, BHN 21,HiH 06 , Pen 99}.

\paragraph{Road map and contributions} 
We start our study    in  Section \ref{sec:Prelim} with the necessary tools  from convex and variational analysis. 
In Section \ref{sec:Kcvx}, we formally introduce and expand on  the central notions of $K$-convexity and $K$-closedness. In particular, in Sections \ref{sec:0cvx}-\ref{sec:PC} we characterize the functions which are convex w.r.t.
a given subspace, half-space and polyhedral cone respectively. In Section \ref{sec:SmallCvx},  we elaborate on Pennanen's characterization of the dual cone of the  smallest closed, necessarily convex (Proposition \ref{prop:NecK}) cone with respect to which a given $F$ is  convex. We extend this in Section \ref{sec:SmallClosed} to study the smallest,  necessarily closed and convex  (Proposition \ref{prop:NecK}) cone with respect to which $F$ is convex {\em and} closed. Section \ref{sec:Main} is fully devoted to the 
question as to when \eqref{eq:Q} holds. Theorem \ref{t:rec_cnd} in Section \ref{sec:Char} provides a characterization for   \eqref{eq:Q}, which consitutes one of the main workhorses for the the rest of Section \ref{sec:Main}. Section \ref{sec:Nec} is mainly devoted to necessary conditions for \eqref{eq:Q}.  Partly  mimicking the scalar case (Theorem \ref{p:epi=conv-gph-1Dcase}),  in Section \ref{sec:Aff} we present necessary conditions  based on affine $K$-minorization and $K$-majorization. Section \ref{sec:Suff}, in turn,   provides sufficient conditions. Section \ref{sec:Ex} presents different examples of $K$-convex maps  by which we illustrate the theory developed in Section \ref{sec:Kcvx} and, more importantly, Section \ref{sec:Main}. In particular, we apply our findings to the following maps:
\begin{itemize}
\item $F:X\in \R^{n\times m}\mapsto \frac{1}{2}XX^T\in \bS^n$;
\item $F:X\in \bS^n_{++}\to X^{-1}\in \bS^n$\quad (inverse matrix);
\item $F:X\in \bS^n\mapsto \lambda(X)\footnotemark\in\R^n$\footnotetext{Here $\lambda(X)=(\lambda_1,\dots,\lambda_n)^T$ is the vector of eigenvalues of $X\in\bS^n$ in decreasing order.}\quad (spectral map);
\item $F:\bE\to \R^m$ where $F_i$ is convex for all $i=1,\dots,m$ \quad (component-wise convex).
\end{itemize}

\noindent
Most of the criteria worked out in the previous sections for the validity of \eqref{eq:Q} are brought to bear directly or indirectly here.

Section \ref{sec:gF} taps into the composite framework alluded to above,  where, primarily, we  study the following question: given a vector-valued map $F$ and  a (closed, proper) convex function $g$ such that $g\circ F$ is convex, does there exist a cone $K$ such that $F$ is $K$-convex and $g$ is increasing in a $K$-related ordering?

\medskip
\noindent
{\em Notation:} In what follows, $\bE$ denotes a Euclidean space, i.e. a finite-dimensional real inner product space with inner product denoted by $\ip{\cdot}{\cdot}$.    Given a set $S\subset\bE$, we denote its closure, convex hull,  closed convex hull,  and convex conical hull by $\cl S,\, \co S, \,\clco S$ and $\cone S$, respectively. For a vector $u\in\bE$, we denote its (convex) conical hull by $\R_{+}u$, and $\R_{++}u=\set{\lambda u}{\lambda>0}$. The indicator function   $\delta_S:\bE\to \rp$ of $S\subset \bE$  is given by $\delta_S(x)=0$ if $x\in S$ and $\delta_S(x)=+\infty$ otherwise.

\section{Preliminaries} \label{sec:Prelim}

Throughout we   make use of the {\em relative topology}  for  convex sets \cite[\S 6]{Roc 70}. The {\em relative interior}  $\ri C$ of a convex set $C\subset \bE$ is its interior in the subspace topology induced by its {\em affine hull} $\aff C:=\set{\lambda x+(1-\lambda)y}{\lambda \in \R,\; x,y\in C}$.  For a convex  set $C\subset\bE$ and (any) $x_0\in C$,   the {\em subspace parallel to} $C$ is $\para C:=\aff C  -x_0$. For convex sets,  we have a handy description of the affine hull. 
\begin{lemma}
\label{l:affine-hull-of-a-convex-set}
Let $C\subset \bE$ be a convex. Then 
$\aff C=\{\alpha x-\beta y \mid \alpha,\beta \geqslant 0, \alpha-\beta =1, x,y\in C\}.
$
\end{lemma}
\begin{proof} Set $A:=\{\alpha x-\beta y \mid \alpha,\beta \geqslant 0, \alpha-\beta =1, x,y\in C\}$. Thus  $\aff C\supset A$. 
Conversely, for  $z\in\aff C $,  there exist $\lambda\in\bR$ and $x,y\in C$ such that $z=\lambda x +(1-\lambda)y$.  If $\lambda\in (0,1)$, then $z\in C$   by convexity, and  hence $z=1\cdot z - 0\cdot z\in A$.  If  $\lambda >1$,  set $\alpha := \lambda\geqslant 0$, $\beta :=\lambda-1\geqslant 0$, and we get $\alpha-\beta = 1$ and $z=\alpha x - \beta y\in A$.  Finally, if $\lambda<0$, then $1-\lambda\geqslant 0$, and thus ($\beta:=\lambda, \alpha:=1-\lambda$) $z\in A$.
\end{proof}

\noindent
Let $f:\bE\to \rbar:=\R\cup\{\pm \infty\}$. We call $f$ {\em proper} if its {\em domain} $\dom f:=\set{x\in\bE}{f(x)<+\infty}$ is nonempty and $f$ doesn't take the value $-\infty$. We say that $f$ is {\em convex} if its {\em epigraph}  $\epi f:=\set{(x,\alpha)\in \bE\times \R}{f(x)<\alpha}$ is convex which  coincides with the usual definition via a  secant condition 
\[
f(\alpha x+(1-\alpha)y)\leq \alpha f(x)+(1-\alpha)f(y)\quad \forall x,y\in \dom f,\; \alpha\in (0,1),
\]
if $f$ does not take the value $-\infty$.
Although  pointwise convergence is not a suitable for preservation of many variational properties, see e.g. \cite[Chapter 7]{RoW 98}, it still preserves convexity in the limit. 

\begin{lemma}\label{lem:Ptw} Let $\{f_k:\bE\to\rp\}$ converge pointwise to $f:\bE\to\rbar$, i.e. $f_k(x)\to f(x)$ for all $x\in\bE$. If $f_k$ is convex for all $k\in\bE$ (sufficiently large), then so is $f$. 
\end{lemma}
\begin{proof} For  $\alpha\in(0,1)$ and $x,y\in\bE$, the convexity of $f_k$  yields
$
\alpha f_k(x)+(1-\alpha)f_k(y) \geqslant f_k(\alpha x+(1-\alpha)y).
$
Passing to the   limit  $k\to\infty$ on  both sides  gives
$
\alpha f+(1-\alpha)f(y) \geqslant f(\alpha x+(1-\alpha)y).
$
\end{proof}

\noindent
We call  $f:\bE\to\rp$  {\em closed} or {\em lower semicontinuous (lsc)} if $\epi f$  is closed.  
We set
\begin{eqnarray*}
\Gamma(\bE)&:=& \set{f:\bE\to\rp}{f\; \text{proper, convex}},\\
\Gamma_0(\bE) &:=& \set{f\in\Gamma(\bE)}{f\;\text{closed}}.
\end{eqnarray*}
For $f\in \Gamma(\bE)$, its {\em closure}  $\cl f\in\Gamma_0(\bE)$ is defined via $\cl (\epi f)=\epi (\cl f).$
More generally, given a convex  subset $D\subset \bE$, we call $f$ {\em $D$-closed} if $\epi f$ is closed in the subspace  topology induced by  $D\times \R$.
We note that $D\times \R$ is a metric space, hence closedness is sequential closedness and, in particular,  $f$ is $D$-closed if and only if 
\[
 \liminf_{k\to\infty} f(x_k)\geq f(\bar x)\quad \forall \{x_k\in D\}\to \bar x\in D.
\]
We define
$
\Gamma_0(D):=\set{f\in\Gamma(\bE)}{f \;\text{$D$-closed}}.
$ 
\begin{lemma}
\label{l:f(x)=clf(x)-iff-f-in-Gamma_0(D)}
Let $f\in\Gamma(\bE)$. Then  the following are equivalent: 
\begin{itemize}
\item[i)]  $f\in\Gamma_0(\dom f)$;
\item[ii)] $f(x)=(\cl f)(x)$ for all $x\in \dom f$.
\end{itemize}
\end{lemma}

\begin{proof} Since $f\in\Gamma_0(\bE)$, we have 
\begin{eqnarray*}
f\in\Gamma_0(\dom f) & \Longleftrightarrow &\epi f  \mbox{ is closed in } \dom f\times\bR \\
& \Longleftrightarrow& \epi f = \cl(\epi f) \cap \dom f\times \bR \\
& \Longleftrightarrow &\epi f = \epi(\cl f) \cap \dom f\times \bR \\
& \Longleftrightarrow & f(x)=\cl f(x)\quad \forall x\in\dom f.
\end{eqnarray*}
Here the first equivalence is simply the definition of $\Gamma_0(\dom f)$. The second is due to the fact that the closed sets in the $\dom f\times \R$ subspace topology are exactly the intersections of closed sets (in $\bE\times \R$) with $\dom f\times \R$. The third one is clear as $\epi (\cl f)=\cl (\epi f)$, and the fourth one follows from elementary considerations.
\end{proof}

\begin{remark}\label{rem:CloCont}  We point out that $f\in \Gamma_0(\bE)$ implies that $f\in\Gamma_0(\dom f)$,  since the closed set  $\epi f\subset \bE\times \R$ intersected with $\dom f\times \R$ is (trivially) closed in the subspace topology induced by $\dom f\times \R$.  However, the converse statement is not true. Consider for instance $\delta_{(0,1)}\in\Gamma_0((0,1))\setminus\Gamma_0(\bR)$, as $(0,1)\times\bR_+$ is a closed set in the topology induced by $(0,1)\times\bR$, but is not a closed set in $\bR\times\bR$.
\end{remark}

\noindent
 A nonempty subset  $K \subset \bE$ is called a {\em cone} if $\lambda x\in K$ for all $\lambda\geq 0$ and $x\in K$. If the latter only holds for all $\lambda>0$, we call $K$ a {\em pre-cone}.  For instance  if $K$ is a cone, then $\ri K$ is a pre-cone, use e.g. \cite[Corollary 6.6.1]{Roc 70}. Combining this with the  {\em line segment principle} \cite[Theorem 6.1]{Roc 70} and \cite[Theorem 6.3]{Roc 70}, we find the following result.
 
\begin{lemma}
\label{l:ri-of-convex-cone}
Let $K\subset\bE$ be a convex cone. Then $\cl K+\ri K\subset \ri K$.
\end{lemma}

 \noindent
 The {\em  polar cone} of a (pre-)cone $K$ is given by $K^\circ := \{ v\in\bE_2 \mid \forall u\in K \colon \langle u,v\rangle \leqslant 0\},$ and $-K^\circ$ is referred to as the {\em dual cone}.
Recall that $\clco K=(K^\circ)^\circ=:K^{\circ\circ}$  by the bipolar theorem \cite[Corollary 6.21]{RoW 98}, and that the polarity operation is order reversing.  
The  {\em horizon cone} of $C\subset \bE$ is given by  
$
C^\infty := \{ u\in\bE \mid \exists \{t_k\}\downarrow 0, \{x_k\in C\} :\; \lim_{k\to\infty} t_kx_k = u\} .
$
If $C$ is a nonempty  convex set, then $C+C^\infty= \cl C$ and for a cone $K$, we have $K^\infty=\cl K$.

A cone $K\subset \bE$  induces and ordering on $\bE$ via
\begin{equation*}
y\geq_K x :\Longleftrightarrow y-x \in K \quad \forall x,y\in\bE.
\label{eq:general-inequality}
\end{equation*}

\begin{lemma}
\label{lem:l1}
If $K$ is a closed and convex cone of $\bE_2$, then
\[
y\geq _K x\quad\Longleftrightarrow\quad\ip{u}{x}\geqslant\ip{u}{y}\quad \forall u\in -K^\circ.
\]
\end{lemma}
\begin{proof}
By the bipolar theorem \cite[Corollary~6.21]{RoW 98}, we have  $K^{\circ\circ}=K$ and hence
\[
x\geqslant_Ky\;\Longleftrightarrow\; x-y\in K=K^{\circ\circ}\;\Longleftrightarrow\;\ip{u}{x-y}\geqslant 0\quad \forall u\in -K^\circ.
\]
\end{proof}

\noindent
A cone $K\subset\bE$  is said to be {\em pointed} if $K\cap(-K)=\{0\}$. Such a cone induces a partial ordering. 

\begin{lemma}[Ordering induced by a pointed cone]\label{lem:Pointed}  Let $K\subset \bE$ be pointed. Then 
\[
x=y \iff x\geq_K y \AND y\geq_K x.
\]
\end{lemma}
\begin{proof} Straightforward.
\end{proof}

\section{\boldmath{$K$}-convexity and \boldmath{$K$}-closedness}\label{sec:Kcvx}

\noindent
We commence this section with the central definitions of this paper.

\begin{definition}[$K$-epigraphs, $K$-convexity and $K$-closedness]
\label{d:Kconvexity-Kepigraph-K-closed} 
Let $K\subset \bE_2$ be a cone   and let  $F\colon D\subset\bE_1\to\bE_2$. Then the  {\em $K$-epigraph} of $F$  is  given by
\begin{equation}
\Epi{K}{F} = \set{ (x,y)\in D\times\bE_2 }{ F(x)\leq_K y}\subset \bE_1\times \bE_2.
\label{eq:Kepigraph}
\end{equation}
We say that $F$ is
\begin{itemize}
\item[i)]  {\em proper} if $\Epi{K}{F}\neq \emptyset$ (i.e. $D\neq\emptyset$);
\item[ii)] {\em $K$-convex}  if $\Epi{K}{F}$ is convex;
\item[iii)] {\em $K$-closed} if   $\Epi{K}{F}$ is closed.
\end{itemize}

\noindent

\end{definition}
\noindent
For $D\subset \bE_1$ convex and $K\subset\bE_2$ a cone,  we point out that  $F:D\to \bE_2$ is $K$-convex if and only if $K$ is convex and
\[
\alpha F(x) + (1-\alpha) F(y) \geqslant_K F(\alpha x + (1-\alpha)y) \quad \forall x,y\in D,\alpha\in(0,1) .
\]
Moreover, we always have
\begin{equation}\label{eq:Kepigraph-1}
\Epi{K}{F} = \gph F + \lbrace 0 \rbrace \times K.
\end{equation}
This has, in particular,  the following  immediate consequence.

\begin{lemma}
\label{l:inculsion-Kepigraphs}
Let $F\colon D\subset\bE_1\to\bE_2$ be proper, and $K_1\varsubsetneq K_2\subset \bE_2$ be  cones. Then $\Epi{K_1}{F}\varsubsetneq\Epi{K_2}{F}$. In particular, there is at most one cone $K\subset\bE_2$ such that $\Epi{K}{F}=\clco(\gph F)$.
\end{lemma}

\noindent
In the  convex case we can extract the following.

\begin{lemma}\label{lem:CC} Let $F\colon D\subset\bE_1\to\bE_2$ be proper, and let $K_1\subset K_2\subset \bE_2$ be  convex cones.  Then
$\Epi{K_2}{F}=\Epi{K_1}{F}+\{0\}\times K_2$. In particular, if $F$ is $K_1$-convex, then $F$ is $K_2$-convex.
\end{lemma}
\begin{proof} This is due to \eqref{eq:Kepigraph-1} combined with the fact that $ K_1+K_2=K_2$ because $K_1$ and $K_2$  are  convex cones.
\end{proof}

\noindent
Given a cone $K\subset \bE$ and its induced ordering, we attach to $\bE$ a  formal largest element $+\infty_\bullet$ with respect to that ordering, and  set $\bE^\bullet:=\bE\cup \{+\infty_\bullet\}$.  
For  $G\colon\bE_1\to\bE_2^\bullet$  its  \textit{domain} is $\dom G:=\lbrace x\in\bE_1 : G(x)\in\bE_2\rbrace$. The {\em graph } of $G$ is  given by
$
\gph G : = \{ (x,F(x)) \mid x\in\dom F\} .
$
We  adopt the notions in Definition \ref{d:Kconvexity-Kepigraph-K-closed} via the restriction $F:=G_{|\dom G}$. 
We record in the next result that   $K$-closedness and $K$-convexity requires certain conditions about the underlying cone $K$.

\begin{proposition}\label{prop:NecK} Let  $K\subset\bE_2$ be a cone, and  let $F\colon\bE_1\to\bE_2^\bullet$ be proper. Then the following hold:
\begin{itemize}
\item[a)] If $F$ is $K$-closed, then $K$ is closed.
\item[b)] If $F$ is $K$-convex, then $K$ is convex.
\end{itemize}
\end{proposition}

\begin{proof} a)  Let $\{y_k\in K\}\to y$ and pick $x\in\dom F$. Then $(x,F(x)+y_n)\in\Epi{K}{F}$  for all $k\in \bN$  and $(x,F(x)+y_k)\to(x,F(x)+y)\in \Epi{K}{F}$ as $\Epi{K}{F}$ is  closed. Thus, $y\in F(x)+y-F(x)=y\in K$.
\smallskip

\noindent
b)  Let $y_1,y_2\in K, \alpha\in (0,1)$. For $x\in \dom F$ we  hence find $(x,F(x)+y_1)\in\Epi{K}{F}$, and $(x,F(x)+y_2)\in\Epi{K}{F}$. As $\Epi{K}{F}$ is a convex, we have $(x,F(x)+\alpha y_1+(1-\alpha)y_2)\in\Epi{K}{F}$, and consequently $\alpha y_1+(1-\alpha)y_2\in K$. Thus, $K$ is convex.
\end{proof}

\noindent
The following proposition shows that  a function $F\colon\bE_1\to\bE_2^\bullet$ is fully determined by its $K$-epigraph when $K$ is a pointed cone.

\begin{proposition}
Let  $K\subset \bE_2$ be a pointed cone,  and let  $F,G\colon\bE_1\to\bE_2^\bullet$ be  proper. Then  
\[
\Epi{K}{F}=\Epi{K}{G} \iff F=G .
\]
\end{proposition}

\begin{proof}
Suppose that $\Epi{K}{F}=\Epi{K}{G}$. In particular, for all $x\in\dom F$, $(x,F(x))\in\Epi{K}{F}$, so $(x,F(x))\in\Epi{K}{G}$, hence $x\in\dom G$ and $F(x)\geqslant_K G(x)$. Likewise, for all $x\in\dom G$, we have $x\in\dom F$ and $G(x)\geqslant_K F(x)$. Thus $\dom F=\dom G$ and for any $x\in\dom F=\dom G$ we have $F(x)=G(x)$ by Lemma \ref{lem:Pointed}.
\end{proof}

\noindent
Given $F\subset\bE_1\colon D\to\bE_2$ and   $u\in\bE_2$, we define the {\em scalarization} $\langle u,F\rangle\colon\bE_1\to\bR\cup\{+\infty\}$ by
\begin{equation}\label{eq:Scalarize}
\langle u,F\rangle (x) = \left\lbrace \begin{array}{ll}
\langle u,F(x)\rangle , & x\in D , \\
+\infty & \mbox{else.}
\end{array}\right. 
\end{equation}
We adapt this notion for $D=\dom F$ if $F:\bE_1\to\bE_2^\bullet$ where $\bE_2$ is ordered by some cone $K$. 
Equipped with this concept, the following proposition gives a characterization of $\Epi{K}{F}$ (and $\gph F$) via the  epigraphs (and graphs) of the scalarizations $\ip{u}{F}$ for $u\in\Kp$. 

\begin{proposition}\label{prop:EpiChar}
\label{p:cone1}
Let    $K\subset\bE_2$ be a closed and convex cone, and let $F\colon\bE_1\to\bE_2^\bullet$ be proper. Then:
\begin{itemize}
\item[a)] $\Epi{K}{F}  = \bigcap_{u\in \Kp}(\id,\ip{u}{\cdot})^{-1}\big(\epi\ip{u}{F}\big)$;
\item[b)] If  $K$ is pointed, then $\gph F=\bigcap_{u\in \Kp}(\id,\ip{u}{\cdot})^{-1}\big(\gph\ip{u}{F}\big)$. 
\end{itemize}
\end{proposition}
\begin{proof}
We deduce from Lemma~\ref{lem:l1} that
\[
\begin{aligned}
\Epi{K}{F}&=\set{(x, v)}{v\geq_KF(x)}\\
&=\set{(x, v)}{\ip{u}{v}\geq\ip{u}{F(x)}\;\forall u\in \Kp}\\
&=\bigcap_{u\in \Kp}(\id,\ip{u}{\cdot})^{-1}\big(\epi\ip{u}{F}\big).
\end{aligned}
\]
Similarly,  if $K$ is pointed, we obtain
\[
\begin{aligned}
\gph F&=\set{(x, v)}{x\in\bE_1,\; v=F(x)}\\
&=\set{(x,v)}{x\in\bE_1,\; v-F(x)\in K\cap(-K)}\\
&=\set{(x, v)}{x\in\bE_1,\; v\geq_KF(x)}\bigcap\set{(x, v)}{x\in\bE_1,\; F(x)\geq_Kv}\\
&=\bigcap_{u\in \Kp}(\id,\ip{u}{\cdot})^{-1}\big(\gph\ip{u}{F}\big).
\end{aligned}
\]
\end{proof}

\noindent
As an immediate consequence of the latter theorem, one obtains  Pennanen's sufficient condition for $K$-closedness  \cite[Lemma 6.2]{Pen 99}, which unfortunately excludes functions with domains that are not closed. We therefore provide the following stronger version whose proof is simply 
a refinement of Pennanen's in the next result's part b). Part a) is a refinement of the  scalarization characterization of $K$-convexity.
 
\begin{proposition} \label{p:lsc1}
Let $K\subset \bE_2$ be a closed, convex cone, let $F:\bE_1\to\bE_2^\bullet$ be proper. Then:
\begin{itemize}
\item[a)] The following are equivalent: 
\begin{itemize}
\item[i)] $\ip{u}{F}$ is convex for all $u\in \ri(-K^\circ)$;
\item[ii)] $\ip{u}{F}$ is convex for all $u\in -K^\circ$;
 \item[iii)] $F$ is $K$-convex. 
 \end{itemize}
\item[b)]   $F$ is $K$-closed if $\ip{u}{F} $ is lower  semicontinuous for all $u\in -K^\circ\setminus\{0\}$ and $K\neq \bE_2$.  
\end{itemize} 
In particular, if $K\neq \bE_2$ and $\ip{u}{F}\in \Gamma_0(\bE_1)$ for all $u\in -K^\circ\setminus\{0\}$, then $F$ is $K$-closed and $K$-convex.
\end{proposition}
\begin{proof}  a) 'i)$\Rightarrow$ ii)': Let $u\in -K^\circ\setminus\{0\}$. Then $u$ is a limit $\{u_k\in \ri(-K^\circ)\}\to u$, and hence $\ip{u}{F}$ is a pointwise limit  of conxex functions $\ip{u_k}{F}$, hence convex by Lemma \ref{lem:Ptw}. 
\\
'ii)$\Rightarrow$ iii)':  Follows from Proposition  \ref{prop:EpiChar} a). 
\\
'iii)$\Rightarrow$ i)': Follows from Lemma \ref{lem:l1}.
\smallskip

\noindent
b) Assume that $F$ is not $K$-closed i.e. there exists $\{(x_k,y_k)\in \Epi{K}{F}\}\to (x,y)\notin \Epi{K}{F}$. Then $x\notin \dom F$    or 
$y-F(x)\notin K=K^{\circ\circ}$. In the latter case  there exists $u^*\in -K^\circ$ such that
\begin{equation}\label{eq:IEK}
\ip{u^*}{y}<\ip{u^*}{F(x)}\AND \ip{u^*}{y_k}\geq \ip{u^*}{F(x_k)}\quad\forall k\in \bN.
\end{equation}
If $x\in \dom F$, then necessarily $u^*\neq 0$. On other hand, if $x\notin \dom F$, since $-K^\circ \supsetneq \{0\}$ by assumption, we can choose
$u^*\neq 0$. All in all, there exists $u^*\in -K^\circ\setminus\{0\}$ such that \eqref{eq:IEK} holds. We hence obtain
\[
\ip{u^*}{F}(x)>\ip{u^*}{y}=\liminf_{k\to\infty}\ip{u^*}{y_k}\geq \liminf_{k\to\infty} \ip{u^*}{F}(x_k),
\]
and, consequently,  $\ip{u^*}{F}$ is not lsc, which concludes the proof of part b). 
\end{proof}

\noindent
We close out this  preparatory paragraph with the following useful result.

\begin{lemma}
\label{l:intersection-cone-K-convexity}
Let  $D\subset\bE_1$  be  nonempty,  let  $F\colon D\to\bE_2$, and let  $(K_i)_{i\in I}$ be a family of cones of $K_i\subset \bE_2$. Then
\[
\Epi{\left(\underset{i\in I}\bigcap K_i\right)}{F} = \underset{i\in I}\bigcap\left( \Epi{K_i}{F}\right).
\]
In particular, if  $F$ is $K_i$-closed for all $i\in I$, then $F$ is $(\cap_{i\in I} K_i)$-closed. Moreover, if    $F$ is $K_i$-convex for all $i\in I$ then $F$  is $(\cap_{i\in I} K_i)$-convex.  The latter is an equivalence if $K_i$ is convex for all $i\in I$.
\end{lemma}

\begin{proof}
For any $x\in\dom F$ and $y\in\bE_2$, we have
\begin{eqnarray*}
(x,y)\in\Epi{\left(\bigcap_{i\in I} K_i\right)}{F} & \Longleftrightarrow  & y-F(x)\in\bigcap_{i\in I} K_i \\
 &\Longleftrightarrow& \forall i\in I:\;y-F(x)\in K_i \\
 &\Longleftrightarrow  &\forall i\in I: \;(x,y)\in\Epi{K_i}{F} \\
 & \Longleftrightarrow  &(x,y)\in \bigcap_{i\in I}\Epi{K_i}{F}.
\end{eqnarray*}
The addendum follows from the fact that  intersection preserves closedness and convexity, and that $\bigcap_{i\in I}K_i\subset K_i$, so $\bigcap_{i\in I}K_i$-convexity implies  $K_i$-convexity for all $i\in I$ if these are convex, see Lemma \ref{lem:CC}.
\end{proof}

%
%

\subsection{Affine and \boldmath{$\{0\}$}-convex functions}\label{sec:0cvx}

We extend the notion of affine functions to affine subsets of $\bE$, see, e.g., Rockafellar \cite[\S 1]{Roc 70} for the standard case.

\begin{definition}
Let $A\subset\bE$ be an affine set  and let $x_0\in A$. Then a function $F\colon A\to\bE_2$ is said to be affine if   there exists a linear map $L:\para A\mapsto \bE_2$ and a vector $b\in\bE_2$ such that, for all $x\in A$, we have $F(x)=L(x-x_0)+b$ for all $x\in A$.\end{definition}

\begin{lemma}
\label{l:0-convexity-and-affine-functions}
Let $A\subset\bE_1$ be affine. Then  $F\colon A\to\bE_2$ is affine if and only if  
\begin{equation}\label{eq:Aff}
F(tx+(1-t)y)=tF(x)+(1-t)F(y)\quad \forall x,y\in A,\;t\in(0,1).
\end{equation}
\end{lemma}
\begin{proof} Assume first that \eqref{eq:Aff} holds.  Discriminating the three cases $t\in [0,1]$, $t>1$ and $t<0$,   it is straightforward to show that, in fact,  we have
\begin{equation}\label{eq:Aff2}
F(tx+(1-t)y)=tF(x)+(1-t)F(y)\quad \forall x,y\in A,\;t\in\R.
\end{equation}
Now  let $x_0\in A$, i.e. $\para A=A-x_0$, and define $L:\para A\to\bE_2$ by $L(x):= F(x+x_0)-F(x_0)$. 
Using \eqref{eq:Aff2}, we   find that  $L(tx+(1-t)y)=tL(x)+(1-t)L(y)$ for all $x,y\in\para A$ and  $t\in\bR$.
 Thus, taking $y=0$, as $L(0)=0$, gives  $L(tx)=tL(x)$ for all $x\in \para A$ and all $t\in\bR$. Hence, $L(x+y)=L(\frac{1}{2} (2x) +\frac{1}{2}(2y))=\frac{1}{2}L(2x)+\frac{1}{2}L(2y)=L(x)+L(y)$, for all $x,y\in \para A$. This implies that $L$ is linear. Hence,  for all $x\in A$ and $b:=F(x_0)$, we have
$
F(x)=L(x-x_0)+b.
$
Thus, $F$ is affine.

 Conversely, if $F$ is affine, then we can write 
$F = L((\cdot)-x_0)+b$  for some linear map $L:\para A\to\bE_2$, $x_0\in A$ and $b\in\bE_2$. Then
\[
F(tx+(1-t)y) =  t(L(x-x_0)+b) + (1-t) (L(y-x_0)+b) = tF(x)+(1-t)F(y),
\]
for all $x,y\in A$ and $t\in\bR$. In particular, this is true for all $t\in (0,1)$.
\end{proof}

\begin{proposition}
\label{p:0-convexity-and-affine-functions}
Let $D\subset \bE_1$ be nonempty convex. Then  $F\colon D\to\bE_2$ is $\{0\}$-convex if and only if  there exists an affine function $G\colon\aff D\to\bE_2$ such that $F=G_{|D}$.
\end{proposition}
\begin{proof}  First, assume that $F\colon D\to\bE_2$ is $\{0\}$-convex and   $z\in\aff D$. 
By Lemma \ref{l:affine-hull-of-a-convex-set}, we can write $z=\alpha x-\beta y$, for some $\alpha,\beta\geqslant 0$, $\alpha-\beta =1$, and $x,y\in D$. Suppose $z$ has two representations of this form, i.e. $z=\alpha x-\beta y=\alpha' x'-\beta' y'$, with $\alpha,\beta,\alpha',\beta'\geqslant 0$, $\alpha-\beta=\alpha'-\beta' =1$, and $x,y,x',y'\in D$. 
Then  $\Delta:=\alpha+\beta'(=\alpha'+\beta)=1+\beta+\beta' >0$. 
By convexity of $D$, we find that 
$
\frac{\alpha}{\Delta}x+\frac{\beta'}{\Delta}y'=\frac{\alpha'}{\Delta}x'+\frac{\beta}{\Delta}y\in D.
$
Using the $\{0\}$-convexity of $F$, we have
$
\frac{\alpha}{\Delta}F(x)+\frac{\beta'}{\Delta}F(y') = \frac{\alpha'}{\Delta}F(x')+\frac{\beta}{\Delta}F(y).
$
Multiplying by $\Delta$ and rearranging the above terms, we get
\[
\alpha F(x)-\beta F(y) = \alpha'F(x')-\beta'F(y') .
\]
Therefore, we the function  $G:\aff D\to\bE_2, \;G(z)=\alpha F(x)-\beta F(y)$ for $z\in\aff D$ given by  $z=\alpha x-\beta y$, $\alpha,\beta\geqslant 0$, $\alpha-\beta=1$, $x,y\in D$ is well-defined.

Now, let $z$ and $z'$ in $D$ given by $z=\alpha x-\beta y$ and $z'=\alpha' x'-\beta y'$, with $\alpha,\beta,\alpha',\beta'\geqslant 0$, $\alpha-\beta =1$, $\alpha'-\beta' =1$, and $x,y,x',y'\in D$. Let $t\in (0,1)$ and set $p:=(1-t)z+tz'$, as well as 
$a:= (1-t)\alpha+t\alpha' $ and $b:=(1-t)\beta + t\beta' $. Then $a,b\geqslant 0$ and $a-b=1$. 

If $b=0$, then $\beta=\beta'=0$ and $\alpha=\alpha'=1$, so $z=x\in D$ and $z'=x'\in D$, and thus,  using the $\{0\}$-convexity of $F$,  we have $G(p)=F(p)=F((1-t)z+tz')=(1-t)F(z)+tF(z')=(1-t)G(z)+tG(z')$.

If $b\neq 0$, then $\beta\neq 0$ or $\beta'\neq 0$, hence $a,b>0$. Then, we have
\[
p = a \underbrace{\left[\frac{(1-t)\alpha}{a}x+\frac{t\alpha'}{a}x'\right]}_{\in D} - b \underbrace{\left[\frac{(1-t)\beta}{b}y+\frac{t\beta'}{b}y'\right]}_{\in D}.
\]
Using this, and recalling the fact that  $a,b\geqslant 0$ and $a-b=1$, the definition of $G$ yields
\[
G(p)= a F\left(\frac{(1-t)\alpha}{a}x+\frac{t\alpha'}{a}x'\right) - b F\left(\frac{(1-t)\beta}{b}y+\frac{t\beta'}{b}y'\right) .
\]
As $F$ is $\{0\}$-convex, we thus infer 
\begin{align*}
G(p) & = a \frac{(1-t)\alpha}{a} F(x) + a \frac{t\alpha'}{a} F(x') - b\frac{(1-t)\beta}{b}F(y) -b \frac{t\beta'}{b}F(y') \\
& = (1-t) \underbrace{[\alpha F(x) -\beta F(y)]}_{G(z)} + t\underbrace{[\alpha' F(x') -\beta' F(y')]}_{G(z')} \\
& = (1-t) G(z) +t G(z') .
\end{align*}
All in all,   by Lemma \ref{l:0-convexity-and-affine-functions}, $G$ is affine.

Conversely, if there exists $G:\aff D\to \bE_2$ affine such that $F=G_{|D}$,  then Lemma \ref{l:0-convexity-and-affine-functions} yields
 $G(tx+(1-t)y)=tG(x)+(1-t)G(y)$  for all $t\in (0,1)$ and $x,y\in D$.
Hence
$F(tx+(1-t)y)=tF(x)+(1-t)F(y)$ for all $x,y\in D$, $t\in (0,1)$, and as $D$ is a convex set,  $F$ is $\{0\}$-convex.

\end{proof}

\noindent
As a simple corollary we  get the following result.

\begin{corollary} The $\{0\}$-convex functions $\bE_1\to \bE_2$ are exactly the affine functions. 
\end{corollary}

\subsection{Convexity with respect to a (nontrivial) subspace}

Using our study on $\{0\}$-convexity above, we are now in a position  to investigate the functions which are convex w.r.t. a given (nontrivial) subspace. To this end, observe that for a subspace $U\subset \bE$, the polar (and the dual) cone of $U$ equal the orthogonal complement $U^\perp:=\set{u\in\bE}{\ip{u}{y}=0\;\forall y\in U}$.

\begin{lemma}
\label{l:convexity-wrt-subspace}
Let $U\subset \bE_2$ be a nontrivial subspace and let $\{e_1,\dots,e_r\}$ be a basis of $U^\perp$. 
Then for $F:D\subset \bE_1\to\bE_2$, with $D\subset \bE_1$ convex,  the following are equivalent:
\begin{enumerate}
\item[i)] $F$ is $U$-convex;
\item[ii)]  $\ip{u}{F}$ is $\{0\}$-convex  for all $u\in U^\perp$;
\item[iii)]  $\ip{e_i}{F}$ is $\{0\}$-convex for all $i=1,\dots,r$.

\end{enumerate}
\end{lemma}

\begin{proof}`i)$\Leftrightarrow$ii)':  
If $F$ is $U$-convex and since $U^\perp$ is a subspace, by  Proposition \ref{p:lsc1}   we find  that  both  $\ip{u}{F}$ and  $\ip{-u}{F}$ are convex for every $u\in U^\perp$. Therefore, $\ip{u}{F}$ is $\{0\}$-convex for all $u\in U^\perp$. In turn, if 
$\ip{u}{F}$ is $\{0\}$-convex  for all $u\in U^\perp$,  then in particular $\ip{u}{F}$ is ($\R_+$-)convex  for all $u\in U^\perp$, so we can use Proposition \ref{p:lsc1} to conclude that $F$ is $U$-convex.
\smallskip

\noindent
`ii)$\Leftrightarrow$iii)':  The implication  ii)$\Rightarrow$iii) is trivial.
In turn, if iii) holds,  let $u\in U^\perp$, i.e. $u=\sum_{i=1}^r u_i e_i$ for some $u_i\in\bR$.  By assumption, $\ip{e_i}{F}\; (i=1,\dots,r)$ are $\{0\}$-convex.  Since   $\ip{u}{F}(x)=\sum_{i=1}^r  u_i\ip{e_i}{F}(x)$ for all $x\in \dom F$,  we find that $\ip{u}{F}$ is  $\{0\}$-convex.
\end{proof}

\begin{proposition}
Let $U\subset \bE_2$ be a nontrivial subspace  and let  $F:D\subset \bE_1\to\bE_2$ with $D$ convex. Then $F$ is $U$-convex if and only if  there exists an affine map $G\colon\aff D\to\bE_2$ and  a  function $H\colon\aff D\to U$ such that $G_{|D}+H_{|D}=F$.
\end{proposition}

\begin{proof}
Suppose that $F$ is $U$-convex and   let $p:\bE_2\to U$ be  the orthogonal projection onto $U$. Define 
\[
H:\aff D\to U, \quad H(x):=\begin{cases} 
p(F(x)), & x\in D,\\
0,&\text{else}.
\end{cases}
\] 
Set $\hat G:=F-H_{|D}$. By Lemma \ref{l:convexity-wrt-subspace} and Proposition \ref{p:0-convexity-and-affine-functions}, $\ip{u}{F}_{|D}$  is the restriction of an affine function $f^u:\aff D\to\R $  for all $u\in U^\perp$. However, as $H(x)\in U\;(x\in \bE_1)$, we have $\ip{u}{\hat{G}}=\ip{u}{F}$ for all $u\in U^\perp$. Now, let  $\{e_1,\dots,e_r\}$ be an orthogonal basis of $U^\perp$. Then $\hat{G}  = \sum_{i=1}^r f^{e_i}_{|D} e_i $. Then  $G := \sum_{i=1}^r f^{e_i} e_i:\aff D\to \bE_2$ is affine and  $F=G_{|D}+H_{|D}$.

Conversely, suppose that there exists a function $H\colon\aff D\to U$ and an affine map $G\colon\ \aff D\to\bE_2$ such that $F = G_{|D} + H_{|D}$. Then for  $u\in U^\perp$ we have  $\ip{u}{F}=\ip{u}{G}_{|D}$, and thus $\ip{u}{F}$ is $\{0\}$-convex, as $\ip{u}{G}$ is affine on $\aff D\supset D$ and $D$ is convex. Thus, by Lemma \ref{l:convexity-wrt-subspace}, $F$ is $U$-convex.
\end{proof}

\subsection{Convexity with respect to a half-space and  polyhedral cones}\label{sec:PC}

\noindent
Every (proper) half-space  $H\subset \bE$ that is also a cone is of the form $H=\set{x\in \bE}{\ip{w}{x}\geq 0}$ for some $w\in \bE\setminus\{0\}$. Clearly, $H$ is (polyhedral) convex and closed with dual cone $-H^\circ=\R_+w$.

\begin{proposition}\label{prop:HS} Let $w\in\bE_2\setminus\{0\}$ and let    $H=\{x\in\bE_2\mid\ip{w}{x}\geqslant 0\}$ be the associated half-space. Then for  $F:D\to\bE_2$ with $D\subset \bE_1$ (nonempty) convex we have: 
\begin{itemize}
\item[a)] $F$ is $H$-convex if and only if $\ip{w}{F}$ is convex.
\item[b)]  $F$ is $H$-closed if $\ip{w}{F}$ is lower semicontinuous.

\end{itemize}
\end{proposition}
\begin{proof}a)    By Proposition \ref{p:lsc1} a), $F$ is $H$-convex if and only if $\ip{u}{F}$ is convex   for all $u\in\ri(-H^\circ) = \bR_{++}w$. However, $\ip{tw}{F}$ is convex if and only if $\ip{w}{F}$ is convex for all $t>0$.
\smallskip

\noindent
b) This follows with a similar argument as in a) from Proposition \ref{p:lsc1} b) observing that $H\neq \bE_2$ and 
$-H^\circ\setminus\{0\}=\bR_{++}w$.
\end{proof}

\noindent
 Proposition \ref{prop:HS} combined with Lemma \ref{l:intersection-cone-K-convexity} yields the following result on polyhedral cones.

\begin{corollary} Let $w_1,\dots, w_l\in\bE\setminus\{0\}$ and let $P=\bigcap_{i=1}^l H_i$ with $H_i=\set{x}{\ip{w_i}{x}\geq 0}$. Then for  $F:D\to\bE_2$ with $D\subset \bE_1$ (nonempty) convex we have:
\begin{itemize}
\item[a)] $F$ is $P$-convex if and only if $\ip{w_i}{F}$ is convex for all $i=1,\dots,l$.
\item[b)] $F$ is $P$-closed if  $\ip{w_i}{F}$ is lower semicontinuous for all $i=1,\dots,l$.
\end{itemize}

\end{corollary}

\subsection{The smallest closed cone with respect to which \boldmath{$F$} is convex}\label{sec:SmallCvx}

%
\noindent
The following result  ensures the existence of a smallest closed cone $K_F$ with respect to which $F$ is convex (hence, $K_F$  is also  convex by Proposition \ref{prop:NecK}) and characterizes its dual cone.

\begin{proposition}[The cone $K_F$ and its dual]\label{prop:KF}
Let $D\subset\bE_1$ be nonempty and   convex   and  let $F\colon D\to\bE_2$.  Then the following hold:
\begin{itemize}
\item[a)] There exists  a smallest (nonempty) closed (and convex) cone $K_F\subset\bE_2$  with respect to which $F$ is convex, i.e if $F$ is $K$-convex and $K$ is closed and convex, then $K\supset K_F$.

\item[b)] The dual cone of $K_F$ from a) is given by 
$-K_F^\circ=\lbrace u\in\bE_2 \mid \langle u,F\rangle \text{ is convex}\rbrace .$
\end{itemize}
\end{proposition}

\begin{proof} a)
Define $K_F$ as the intersection of all closed and convex cones which respect to which $F$ is convex. Then $K_F$ is nonempty (as $F$ is $\bE_2$-convex and every cone contains $0$), closed and convex, and, by  Lemma \ref{l:intersection-cone-K-convexity}, $F$ is $K_F$-convex. By construction, there is no smaller cone with these properties. 
\smallskip

\noindent
b)  See  \cite[Lemma 6.1]{Pen 99}.
\end{proof}

%
%
%

\subsection{The smallest (closed) cone with respect to which \boldmath{$F$} is convex {\em and} closed}\label{sec:SmallClosed}

We now investigate how the  situation changes in comparison to the study in Section \ref{sec:SmallCvx} when we are looking for the smallest 
(by Proposition \ref{prop:NecK} necessarily closed and convex) cone with respect to which a function $F$ is convex {\em and} closed.    Such cone   does not need to exist, simply because a given function $F:\bE_1\to\bE_2^\bullet$ need not be $\bE_2$-closed.  More concretely, for every cone $K\subset \R$,  the indicator $F:=\delta_{(0,1]}$ has $\Epi{K}{F}=(0,1]\times K$, which is not closed.

\begin{proposition}[The cone $\hat K_F$]
\label{p:def-of-K_F^hat}
Let $D\subset\bE_1$  be nonempty and   convex  and let  $F\colon D\to\bE_2$  such that there exists a (necessarily closed  and convex) cone $K\subset\bE_2$ with respect to which $F$ is closed and convex. Then  there exists a smallest closed and convex cone $\hat{K}_F\subset \bE_2$ such that $F$ is $\hat{K}_F$-closed and $\hat{K}_F$-convex.
\end{proposition}

\begin{proof} Follows readily from   Lemma \ref{l:intersection-cone-K-convexity}.
\end{proof}

\noindent
In the spirit of Proposition \ref{prop:KF} b), we want to characterize the dual cone of $\hat K_F$ (if it exists). To this end, the following lemma is useful.

\begin{lemma}\label{lem:Kbar} Let $F:D\subset\bE_1\to\bE_2$ with $D$ convex (and nonempty). Then  the following hold:
\begin{itemize}
\item[a)] The set
$
\cl\set{u\in \bE_2}{\ip{u}{F}\in \Gamma_0(\bE_1)} 
$
is either empty or a closed convex cone. 
\item[b)]  The set $\set{u\in \bE_2}{\ip{u}{F}\in \Gamma_0(D)}$ is a convex cone (in particular nonempty).
\end{itemize}
\end{lemma}
\begin{proof}a) Assume nonemptiness, in which case  $S:=\set{u\in \bE_2}{\ip{u}{F}\in \Gamma_0(\bE_1)}$ is also nonempty. Then, clearly,  $S$ is convex and, consequently, so is $\cl S$. Moreover, for  $u\in \cl S$ and $\lambda\geq  0$ there exist $\{u_k\in S\}\to u$ and $\{\lambda_k>0\}\to \lambda$ and with $\lambda_k u_k\in S$, hence $\lambda u\in \cl S$, and thus $\cl S$ is a closed convex cone.
\smallskip

\noindent
b) Set $K:=\set{u\in \bE_2}{\ip{u}{F}\in \Gamma_0(D)}$. We first show that $K$ is a cone. To this end, note that $\ip{0}{F}=\delta_{D}$, whose epigraph $D\times \R_+$ is clearly closed  in the topology induced by $D\times \R$. Now, for $u\in K$  and  $\lambda>0$ observe that 
$\epi \ip{\lambda u}{F}=L_\lambda^{-1}(\epi \ip{u}{F})$ for $L_\lambda:(x,\alpha)\mapsto (x,\alpha/\lambda)$, which is then closed as the linear preimage of a closed set in the topology induced by $D\times \R$. Therefore, $K$ is a cone. 
In order  prove that $K$ is convex, it hence suffices to show that $K+K\subset K$, see \cite[Exercise 3.7]{RoW 98}.  To this end, let $u,v\in K$ and take $\{(x_k,\alpha_k)\in \epi \ip{u+v}{F}\}\to (x,\alpha)\in D\times \R$.  In particular, we have $(x_k,\alpha_k-\ip{v}{F(x_k)})\in \epi \ip{u}{F}$ for all $k\in\bN$. Let $z:=\liminf_k \ip{v}{F(x_k)}$.  W.l.o.g. $z=\lim_k \ip{v}{F(x_k)}$ (otherwise go to subsequence), and by $D$-closedness of $\ip{v}{F}$, we find $(x,z)\in \epi \ip{v}{F}$, i.e. $\ip{v}{F(x)}\leq z$. Moreover, by $D$-closedness of $\ip{u}{F}$ we find that $(x,\alpha-z)\in \epi \ip{u}{F}$, hence
$\alpha\geq \ip{u}{F(x)}+z\geq \ip{u}{F}(x)+ \ip{v}{F}(x)$. Therefore, $(x,\alpha)\in \epi \ip{u+v}{F}$.
\end{proof}

\begin{proposition}[The dual cone of $\hat K_F$]
\label{p:dual-cone-of-K_F^hat}
Let  $F\colon D\subset\bE_1\to\bE_2$ with $D$ nonempty and  convex. If  $\hat{K}_F$ (in the sense of Proposition \ref{p:def-of-K_F^hat}) exists, we have
\begin{equation}
-\hat{K}_F^\circ=\cl\lbrace u\in\bE_2 \mid \langle u,F\rangle \in\Gamma_0(\bE_1)\rbrace   =\cl \lbrace u\in\bE_2 \mid \langle u,F\rangle \in\Gamma_0(
D)\rbrace .
\end{equation}
\end{proposition}

\begin{proof}  Set $\bar K:=\cl(\lbrace u\in\bE_2 \mid \langle u,F\rangle \in\Gamma_0(\bE_1)\rbrace  ).$  To prove the first identity, use
  \cite[Corollary 7.4(ii)]{Pen 99}, to find $\langle u,F\rangle\in\Gamma_0(\bE_1)$ for all all $u\in\ri (-\hat{K}_F^\circ)$. Consequently, also using \cite[Theorem 6.2]{Roc 70}, we have  $-\hat{K}_F^\circ\subset\bar K$. 
Now, assume that the converse inclusion were false, and consequently (by what was already proved) $-\hat{K}_F^\circ\subsetneq\bar K$.  Then, by definition of $\bar K$, there must acutally exist $u_0\in \lbrace u\in\bE_2 \mid \langle u,F\rangle \in\Gamma_0(\bE_1)\rbrace\setminus (-\hat{K}_F^\circ)$ (since otherwise $-\hat{K}_F^\circ=\bar K$, in contradiction to the assumption). In particular, $u_0\neq 0$ and,   by positive homogeneity, 
we have $\ip{tu_0}{F}\in\Gamma_0(\bE_1)$ for all $t>0$, thus $\ip{u}{F}\in\Gamma_0(\bE_1)$ for all $u\in \R_+u_0\setminus \{0\}$. Therefore, by Proposition \ref{p:lsc1}, with $L:=-(\R_+u_0)^\circ\subsetneq \bE_2$, we find that  $F$ is $L$-convex and $L$-closed. Thus, by Lemma \ref{l:intersection-cone-K-convexity}, we find that  $F$ is $(\hat{K}_F\cap L)$-closed and -convex. However, since $u_0\notin -\hat K_F^\circ$, we cannot have $L\supset \hat K_F$, and hence $\hat{K}_F\cap L\subsetneq \hat{K}_F$. This contradicts the definition of $\hat{K}_F$ as the smallest closed, convex cone with respect to which $F$ is convex and closed, and thus we have proved the first identity.

To prove the second identity  we first note  that, from what was proved above and by Remark \ref{rem:CloCont}, we have 
$
-\hat{K}_F^\circ = \cl(\lbrace u\in\bE_2 \mid \langle u,F\rangle \in\Gamma_0(\bE_1)\rbrace  ) \subset \cl\lbrace u\in\bE_2 \mid \langle u,F\rangle \in\Gamma_0(D)\rbrace .
$
In order to prove the converse inclusion, set $K:=-\lbrace u\in\bE_2 \mid \langle u,F\rangle \in\Gamma_0(
D)\rbrace^\circ$, and observe that  by Lemma \ref{lem:Kbar} b) and  \cite[Theorem 6.3]{Roc 70}, respectively, we have
\begin{equation}\label{eq:ConeCont}
-K^\circ=\cl\lbrace u\in\bE_2 \mid \langle u,F\rangle \in\Gamma_0(D)\rbrace\AND \ri(-K^\circ)\subset \lbrace u\in\bE_2 \mid \langle u,F\rangle \in\Gamma_0(D)\rbrace.
\end{equation}
We now show that $F$ is $K$-convex: To this end, first note that  $\ip{u}{F}$ is convex for all $u\in\ri(-K^\circ)$,   by \eqref{eq:ConeCont}.
Any $u\in-K^\circ$ is a limit $\{u_k\in \ri(-K^\circ)\}\to u$, and therefore $\ip{u}{F}$ is the pointwise limit of convex functions $\ip{u_k}{F},$ and hence convex (Lemma \ref{lem:Ptw}). Thus, by Proposition \ref{p:lsc1} a), $F$ is $K$-convex.

We now prove that $F$ is also $K$-closed: To this end, let $\{(x_k,y_k)\in \Epi{K}{F} \}\to (x,y)$. In particular, there exists $\{v_k\in K\}$ such that
$F(x_k)+v_k=y_k$ for all $k\in \bN$. Moreover, as $K\subset \hat K_F$ (see above)  and $\Epi{\hat K_F}{F}$ is closed (by definition),  we have $(x,y)\in\Epi{\hat K_F}{F}$, and consequently,  $x\in D$. 
Thus we can use the fact that, by \eqref{eq:ConeCont}, $\ip{u}{F}$ is $D$-closed for all $u\in \ri(-K^\circ)$, and hence
\[
\ip{u}{y}=\lim_{k\to\infty} \ip{u}{y_k} = \lim_{k\to\infty} \ip{u}{F(x_k)}+\underbrace{\ip{u}{v_k}}_{\geqslant 0}\geqslant \liminf_{k\to\infty} \ip{u}{F}(x_k) \geqslant \ip{u}{F}(x) 
\]
for all $u\in \ri(-K^\circ)$. Now, every $u\in -K^\circ$ is a limit $\{u_k\in -\ri(K^\circ)\}\to u$ with $\ip{u_k}{y}\geq \ip{u_k}{F(x)}$, and hence $\ip{u}{y}\geq \ip{u}{F(x)}$. As $u\in-K^\circ$ was arbitrary, this shows  that $y\geq_K F(x)$, and thus $(x,y)\in \Epi{K}{F}$, which shows that $\Epi{K}{F}$ is closed and hence $F$ is $K$-closed (and $K$-convex as proved earlier). Since $K\subset \hat K_F$ it follows that $\hat K_F=K$, which concludes the proof.
\end{proof}

\noindent
The natural question as to when closures in the above result are superfluous is addressed after the following auxiliary result.

\begin{lemma}
\label{l:ball-in-affine-space} Let
 $\{e_1,\dots,e_n\}\subset\bE$ be an  orthonormal system and $A\subset\bE$  an affine set  such that $\para A=\lin\{e_1,\dots,e_n\}$. Then, for all $r>0$ and all $x\in A$, we have 
$
B_{\frac{r}{n}}(x)\cap A\subset\co\set{x\pm re_i}{i=1,\dots,n}.
 $
\end{lemma}
\begin{proof}
Let $y\in B_\frac{r}{n}(x)\cap A$, i.e.   $y=x+\sum_{i=1}^n y_i e_i$ for some  $y_1,\dots,y_n\in \R\;(i=1,\dots,n)$ 
and 
$
\| y-x \|^2 \leqslant\frac{r^2}{n^2}.
$
As $e_1,\dots,e_n$ is an orthonormal system, we have 
\begin{equation}\label{eq:AuxCvx}
 \sum_{i=1}^n y_i^2=\left\|\sum_{i=1}^n y_i e_i\right\|^2=\|y-x\|^2\leq \frac{r^2}{n^2}.
\end{equation}
On the other hand, we have
\begin{align*}
y & = x +\sum_{i=1}^n y_i e_i \\
& = x+\frac{1}{2}\left(1-\sum_{i=1}^n\frac{|y_i|}{r}\right) \underbrace{(r e_1 -r e_1)}_{=0} + \sum_{i=1}^n \frac{|y_i|}{r}(\mbox{sgn}(y_i)r e_i) \\
& =\frac{1}{2}\left(1-\sum_{i=1}^n\frac{|y_i|}{r}\right) (x+r e_1) + \frac{1}{2}\left(1-\sum_{i=1}^n\frac{|y_i|}{r}\right)(x -r e_1) + \sum_{i=1}^n \frac{|y_i|}{r}(x+\mbox{sgn}(y_i)r e_i). 
\end{align*}
In view of \eqref{eq:AuxCvx},   $|y_i|\leqslant\frac{r}{n}$ for all $i=1,\dots,n$, and thus $y \in \co\set{x\pm re_i}{i=1,\dots,n}$, which concludes the proof.
\end{proof}

\begin{corollary}
\label{c:dual-cone-of-K_F^hat-part-2}
Let  $F\colon D\subset\bE_1\to\bE_2$ with $D$ nonempty and  convex and assume that $\hat K_F$ exists.  Then:
\begin{enumerate}
\item[a)] $-\hat{K}_F^\circ=\{u\in\bE_2\mid\ip{u}{F}\in\Gamma_0(D)\}$ if and only if $\{u\in\bE_2\mid\ip{u}{F}\in\Gamma_0(D)\}$ is closed.
\item[b)] If, for every  sequence $\{x_k\in D\}\to x\in D$ (and every $x\in D$), there exists $v\in\ri(-\hat{K}_F^\circ)$ for which $\{\ip{v}{F}(x_k)\}$ does not tend to $+\infty$, then  $\{u\in\bE_2\mid\ip{u}{F}\in\Gamma_0(D)\}$ is closed.
\item[c)] If,  for all $x\in D\setminus\ri D$, there exists a neighborhood $\cN_x$ of $x$, and a  continuous 
$\hat{K}_F$-majorant\footnote{We refer the reader to Definition \ref{def:KMin} for a formal introduction.} of $F$  on $\cN_x\cap D$, then $\{u\in\bE_2\mid\ip{u}{F}\in\Gamma_0(D)\}$ is closed. 
\end{enumerate}
In particular,  $-\hat{K}_F^\circ=\{u\in\bE_2\mid\ip{u}{F}\in\Gamma_0(D)\}$ if $D$ is relatively open (e.g. affine).

\end{corollary}

\begin{proof} a) Follows readily from Proposition \ref{p:dual-cone-of-K_F^hat}.
\smallskip

\noindent
b) Denote $K=\{u\in\bE_2\mid\ip{u}{F}\in\Gamma_0(D)\}$, which is a convex cone by Lemma \ref{lem:Kbar} b). By Proposition \ref{p:dual-cone-of-K_F^hat}, $-\hat{K}_F^\circ = \cl K $. Consider $u\in\cl K$. Then $ \ip{u}{F}$ is convex by Lemma \ref{lem:Ptw} and   proper with $\dom \ip{u}{F}=D$.

Now let  $x\in D$ and  $\{x_k\in D\}\to x\in D$. Then, by assumption, there exists $v\in\ri(-\hat{K}_F^\circ)=\ri K$ such that $\{\ip{v}{F}(x_k)\}$ is uniformly  bounded away from $+\infty$, hence  w.l.o.g. we can assume that $\ip{v}{F}(x_k)\to r<+\infty$.  As $\ip{v}{F}\in\Gamma_0(D)$, we have  $-\infty<\ip{v}{F}(x)\leq r<+\infty$.  Using Lemma \ref{l:ri-of-convex-cone} (and that $\ri K$ is a pre-cone), we have $u+tv\in\ri K$, and  hence $\ip{u+v}{F}\in\Gamma_0(D)$ for all $t>0$. Thus
\begin{align*}
\liminf_{k\to\infty} \ip{u}{F}(x_k) & = \liminf_{k\to\infty} \ip{u+tv}{F}(x_k) - t\ip{v}{F}(x_k) \\ & =\liminf_{k\to\infty}\ip{u+tv}{F}(x_k) - tr \\ & \geqslant \ip{u+tv}{F}(x)-tr .
\end{align*}
Letting  $t\to 0$ gives $\liminf_{k\to\infty} \ip{u}{F}(x_k)\geqslant\ip{u}{F}(x)$, hence $\ip{u}{F}\in\Gamma_0(D)$ as desired.
\smallskip

\noindent
c) Let $\{x_k\in D\}\to x\in D$.   We will show that   $\ip{v}{F}(x_k)$ does not tend to $+\infty$, for any $v\in\ri(-\hat{K}_F^\circ)$, which by b) then gives the desired conclusion.

First, suppose that  $x\in\ri D$ and set $A:=\aff D $. Let $\{e_1,\dots,e_n\}\subset \bE_1$ be  an orthonormal system  such that $A=x+\lin\{e_1,\dots,e_n\}$. Now, $x\in\ri D $ implies that there exists $\varepsilon >0$  such that $x\pm\varepsilon e_i \in D$ for all $i=1,\dots,n$. As $D$ is  convex, we have $\co\set{x\pm\varepsilon e_i}{i=1,\dots,n}\subset D$. As $x_k\to x$, then $x_k\in B_{\frac{\varepsilon}{n}}(x)$  for $k\in\bN$ large enough. As $x_k\in D$, we  have  $x_k\in A$ for all $k\in\bN$. Hence, , by Lemma \ref{l:ball-in-affine-space},  we have  $x_k\in B_{\frac{\varepsilon}{n}}(x)\cap A\subset\co\set{x\pm\varepsilon e_i}{i=1,\dots,n}\subset D$  for $k\in\bN$ large enough.
Now let $v\in\ri(-\hat{K}_F^\circ)=\ri\{u\in\bE_2\mid\ip{u}{F}\in\Gamma_0(D)\}$. Then, $\ip{v}{F}$ is convex, so for $k$ large enough, we have  $\ip{v}{F}(x_k)\leqslant\max\set{\ip{v}{F}(x\pm\varepsilon e_i)}{i=1,\dots,n}<+\infty$. Hence $\ip{v}{F}(x_k)$ does not tend to $+\infty$. 

In turn, for $x\in D\setminus\ri D$ let $G_x$ be the  continuous $\hat{K}_F$-majorant of $F$ on $\mathcal{N}_x$ and  let $v\in\ri(-\hat{K}_F^\circ)$. Then, as $G_x(y)\geqslant_K F(y)$ for all $y\in \mathcal{N}_x$, hence $\ip{v}{G_x} (y) \geqslant \ip{v}{F} (y)$ for all $y\in \mathcal{N}_x$. Since   $\{x_k\in D\}\to x$, we have that  $x_k\in\mathcal{N}_x$ for $k$ sufficiently large, and thus  $\ip{v}{F}(x_k)\leqslant\ip{v}{G_x}(x_k)$. However, $G_x$ is continuous, so $\ip{v}{G_x}$ is continuous as well,  hence $\ip{v}{G_x}(x_k)\to \ip{v}{G_x}(x)<+\infty$, thus $\ip{v}{F}(x_k)$ does not tend to $+\infty$. 
\end{proof}

\noindent
We close out this section by clarifying the question  as to when $K_F$ and $\hat K_F$ coincide.

\begin{proposition}[$K_F=\hat K_F$]\label{prop:ConeEq} Let $D\subset \bE_1$ be nonempty and convex and let  $F:D\to\bE_2$. Then $K_F=\hat  K_F$ if and only if $F$ is $K_F$-closed. 
\end{proposition}
\begin{proof}By definition of the respective cones,  we always have $\hat K_F\supset K_F$.  But  if $F$ is $K_F$-closed then,  $\hat K_F\subset K_F$, by definition of $\hat K_F$, and hence equality holds. 
 
In turn, if $F$ is not $K_F$-closed, then $K_F\neq\hat{K}_F$, since $F$ is $\hat{K}_F$-closed by definition. 
\end{proof}


\section{When is \boldmath{$\clco(\gph F)=\Epi{K}{F}$}?}\label{sec:Main}

\noindent
This section is  devoted to  our main question as to when the closed convex hull of the graph of a function equals its $K$-epigraph.

\subsection{A  characterization via the horizon cone}\label{sec:Char}

\noindent
We commence this subsection with the central link between the graph and the  $K$-epigraph of a function. To obtain an elegant proof we briefly tap into Fenchel conjugacy \cite{Roc 70}. To this end, realize that every set $S\subset \bE$ is uniquely determined through its indicator function $\delta_S:\bE\to \rp$ which is paired in duality with the {\em support function} $\sigma_S:\bE\to\rp$, $\sigma_S(x)=\sup_{y\in S}\ip{x}{y}$  via the conjugacy relations
$\delta^*_S=\sigma_{S}=\sigma_{\clco S}$, hence $\sigma_S^*=\delta_{\clco S}$, and  thus $\clco \delta_S=\delta_{\clco S}$.

\begin{proposition}\label{p:sup2} Let $K\subset\bE_2$  be a closed, convex cone and let  $F\colon\bE_1\to\bE_2^\bullet$ be proper,  $K$-closed and $K$-convex. Then 
\begin{equation*}
\Epi{K}{F}=\cl\big(\overline{\co}(\gph F)+\{0\}\times K\big).
\end{equation*}
\end{proposition}
\begin{proof} Using  \eqref{eq:Kepigraph-1} and   the set-additivity for support functions  we have
\begin{equation*}\label{eq:supp1}
\sigma_{\Epi{K}{F}}  = \sigma_{\gph F+\{0\}\times K}  = \sigma_{\gph F} + \sigma_{\{0\}\times K}=\sigma_{\clco(\gph F)}+\delta_{\bE_1\times K^\circ}.
\end{equation*}
Moreover, the assumptions on $F$ imply that  $\Epi{K}{F}$ is  closed and convex, and hence
\[
\delta_{\Epi{K}{F}}= \big(\sigma_{\overline{\co}(\gph F)}+\delta_{\bE_1\times K^\circ}\big)^*
=\cl\big(\delta_{\overline{\co}(\gph F)}\infconv \delta_{\{0\}\times K}\big)  =\delta_{\cl\big(\overline{\co}(\gph F)+\{0\}\times K\big)}.
\]
Here the second identity uses  \cite[Theorem~16.4]{Roc 70}, while   the third holds due to the identity $ \delta_{A}\infconv\delta_{B}=\delta_{A+B}$ for any two sets.
\end{proof}

\noindent
We are now in a position to state our first main theorem. 

\begin{theorem}
\label{t:rec_cnd} 
Let $K\subset \bE_2$ be a closed convex cone and let  $F\colon\bE_1\to\bE_2^\bullet$ be $K$-convex and $K$-closed. Then
\[
\Epi{K}{F}=\clco(\gph F)\;\Longleftrightarrow\; \{0\}\times K\subset [\clco(\gph F)]^{\infty}.
\]
\end{theorem}
\begin{proof}
Suppose that $\Epi{K}{F}=\clco(\gph F)$. It follows from Proposition~\ref{p:sup2} that
\[
\clco(\gph F) +\{0\}\times K\subset \cl\big(\overline{\co}(\gph F)+\{0\}\times K\big)=\clco(\gph F).
\]
Taking the horizon cone on both sides  and using  \cite[Exercise 3.12]{RoW 98}, yields $\{0\}\times K\subset [\clco(\gph F)]^{\infty}$. 

Now suppose that  $\{0\}\times K\subset [\clco(\gph F)]^{\infty}$. Then 
\[
\clco(\gph F) +\{0\}\times K\subset \clco(\gph F) +[\clco(\gph F)]^{\infty}=\clco(\gph F).
\]
where the last identity uses e.g. \cite[Theorem 3.6]{RoW 98}. 
Therefore, again using Proposition~\ref{p:sup2},  we obtain
\[
\Epi{K}{F}=\cl(\clco(\gph F)+\{0\}\times K)\subset \clco(\gph F)\subset\Epi{K}{F}.
\]
\end{proof}

\noindent
We will frequently make us of the following trivial observation.

\begin{remark}\label{rem:Drop} We observe that the closure operation in $ [\clco(\gph F)]^{\infty}$ is superfluous, i.e. 
\[
 [\clco(\gph F)]^{\infty}=[\co(\gph F)]^{\infty}.
\]
\end{remark}
\noindent
We immediately obtain the following sufficient condition. 

\begin{corollary}\label{cor:Suff} \label{t:rec_cnd} 
Let $K\subset \bE_2$ be a closed, convex cone and let  $F\colon\bE_1\to\bE_2^\bullet$ be $K$-convex and $K$-closed such that 
$\{0\}\times K\subset \clco(\gph F)$. Then $\Epi{K}{F}=\clco(\gph F)$.
\end{corollary}
\begin{proof} Combine Theorem \ref{t:rec_cnd}  with \cite[Exercise 3.11]{RoW 98} and the fact that the horizon operation preserves inclusion.
\end{proof}

\noindent
Combining  Theorem \ref{t:rec_cnd} with Lemma \ref{l:inculsion-Kepigraphs} yields the following  result.

\begin{corollary}
\label{c:K^F}
Let $K$ be a cone of $\bE_2$ and let $F\colon\bE_1\to\bE_2^\bullet$ be proper, and define the closed, convex cone
$
K^F := \{ u\in\bE_2 \mid (0,u)\in [\clco(\gph F)]^\infty\} .
$
Then
\[
\Epi{K}{F} = \clco(\gph F) \iff K = K^F = \hat{K}_F .
\]
\end{corollary}

\begin{proof} First, let  $K=K^F=\hat{K}_F$.  By definition of $K^F=K$ we hence have $\{0\}\times K\subset [\clco(\gph F)]^\infty$.
From Theorem \ref{t:rec_cnd} we thus  conclude that $\Epi{K}{F}=\clco(\gph F)$. 

In turn, assume that $\Epi{K}{F} = \clco(\gph F)$. Then  by Theorem \ref{t:rec_cnd} we have $K\subset K^F$, and hence $\clco(\gph F) = \Epi{K}{F} \subset \Epi{K^F}{F}$. In addition, $\{0\}\times K^F \subset [\clco(\gph F)]^\infty$, by definition of $K^F$. Hence, using    \eqref{eq:Kepigraph-1}  and the horizon property of convex sets, we have
\[
\Epi{K^F}{F} = \gph F + \{0\}\times K^F\subset \clco(\gph F) + [\clco(\gph F)]^\infty = \clco(\gph F).
\]
Thus, $\Epi{K^F}{F}=\clco(\gph F)=\Epi{K}{F}$ and hence,  by Lemma \ref{l:inculsion-Kepigraphs}, we already have $K^F=K$. Moreover, as $F$ is $K$-convex and $K$-closed, we have $\hat{K}_F\subset K$, thus $\Epi{\hat{K}_F}{F}\subset\Epi{K}{F}= \clco(\gph F)$. Using the fact that $\clco(\gph F)\subset\Epi{\hat{K}_F}{F}$ (as $\Epi{\hat{K}_F}{F}$ is a closed convex set containing $\gph F$), we conclude that $\Epi{\hat{K}_F}{F}=\Epi{K}{F}$, hence, again  by Lemma \ref{l:inculsion-Kepigraphs}, $K=\hat{K}_F$.
\end{proof}

\subsection{Necessary conditions}\label{sec:Nec}

\noindent
In this subsection  we discuss necessary conditions for     $\clco(\gph F)=\Epi{K}{F}$.

\subsubsection{Necessary conditions on the dual cone}

\begin{proposition}
\label{p:K-included-in-dual-cone-of-hatK_F}
Let  $K\subset\bE_2$ be  a cone and $F:\bE_1\to\bE_2^\bullet$ proper such that  $\Epi{K}{F}=\clco(\gph F)$.  Then $K\subset -\hat{K}_F^\circ$.
\end{proposition}

\begin{proof} Let $u\in K\setminus\{0\}$. 
By Theorem \ref{t:rec_cnd}, we have that $(0,u)\in [\clco(\gph F)]^\infty$.  By Remark \ref{rem:Drop} there hence exist $\{(x_k,y_k)\in \co(\gph F)\}$ and $\{\lambda_k\}\downarrow 0$ such that $\lambda_k(x_k,y_k)\to (0,u)$. With $\kappa:=\dim \bE_1\times \bE_2$ and  Carath\'eodory's theorem \cite[Theorem 17.1]{Roc 70}, for  $i=1,\dots,\kappa+1$,  we find sequences $\{x_{k}^{i}\in \dom F\}_k$ and $\{\alpha_{k}^i\}_k$ such that $\sum_{i=1}^{\kappa+1}\alpha_{k}^i=1$ for all $k\in\bN$ as well as
\[
x_k= \sum_{i=1}^{\kappa+1}\alpha_{k}^i x_{k}^{i}\AND y_k=\sum_{i=1}^{\kappa+1}\alpha_{k,j}^i F(x_{k,j}^{i})\quad\forall k\in\bN. 
\]
Now let $t\geq 0$. Then $t_k:=t\frac{\lambda_k}{\|u\|^2}\downarrow 0$ and 
\[
t_k x_k\to 0\AND t_k  \sum_{i=1}^{\kappa+1}\alpha_{k}^i \ip{u}{F}(x_{k}^{i})=t\frac{\ip{\lambda_ky_k}{u}}{\|u\|^2}\to t.
\]
Thus,   for $t\geq 0$, we have $(0,t)\in[\co(\gph\langle u,F\rangle) ]^\infty=[\clco(\gph\langle u,F\rangle) ]^\infty$. Hence $\lbrace 0\rbrace\times\bR_+\subset [\clco(\gph\langle u,F\rangle) ]^\infty$, so by Theorem \ref{t:rec_cnd}, $\epi\langle u,F\rangle = \clco(\gph\langle u,F\rangle)$. This means that $\epi\langle u,F\rangle$ is closed and convex so that $\langle u,F\rangle\in\Gamma_0(\bE_1)$. Therefore $K\setminus \{0\}\in \lbrace u\in\bE_2 \mid \langle u,F\rangle \in\Gamma_0(\bE_1)\rbrace  $, hence $K\subset \cl(\lbrace u\in\bE_2 \mid \langle u,F\rangle \in\Gamma_0(\bE_1)\rbrace  )=-\hat{K}_F^\circ$, by 
 Proposition \ref{p:dual-cone-of-K_F^hat}.
\end{proof}

\noindent
We readily derive the following necessary condition on the dual cone.

\begin{corollary}
\label{c:K-included-in-minus-polar-cone}
Under the assumptions of Proposition \ref{p:K-included-in-dual-cone-of-hatK_F}, we have   $K\subset -K^\circ$.
\end{corollary}

\begin{proof}
By Corollary \ref{c:K^F}, $K=\hat{K}_F$, and by Proposition \ref{p:K-included-in-dual-cone-of-hatK_F}, $K\subset -\hat{K}_F^\circ = -K^\circ$.
\end{proof}

\subsubsection{Affine majorization and minorization}\label{sec:Aff}

For motivational purposes, we start this subsection with the scalar case ($K=\R_+$), where the question whether the closed convex hull of the graph of a function equals its $K$-epigraph can be fully  answered via affine majorization. The proof relies, in essence, on a standard separation argument.

\begin{theorem}[The scalar case]
\label{p:epi=conv-gph-1Dcase}
Let $f\in\Gamma_0(\bE)$. Then  $\epi f = \clco(\gph f)$  if and only if   $f$ does not have an affine majorant on its domain. 
\end{theorem}
\begin{proof}  Suppose  that there exists $(\bar x,\bar t)\in\epi f\backslash\clco(\gph f)$.  In particular,  $\bar x\in\dom f$, and by strong separation  \cite[Corollary 11.4.2]{Roc 70}, 
there exists $(s,\alpha)\in\bE_1\times\R$ such that
\begin{equation}
\label{eq:20190429e1}
\ip{s}{\bar x}+\alpha\bar t >\sup_{(x,t)\in\clco(\gph f)}\ip{s}{x}+\alpha t.
\end{equation}
Choosing   $(x,t):= (\bar x,f(\bar x))$, we find that 
$
\alpha(\bar t -f(\bar x))>0,
$
and hence, $\alpha >0$. It then follows from \eqref{eq:20190429e1} with $x\in\dom f$ and $t=f(x)$ that
$
\ip{s/\alpha}{\bar x-x}+\bar t>f(x).
$ 
Thus $f$ is majorized on its domain by the affine map $x\mapsto -\ip{s/\alpha}{x}+\ip{s/\alpha}{\bar x}+\bar t$, which proves  one direction.

To prove the converse  implication,  suppose now that $f$ has an affine majorant on its domain, i.e.   there exists  $(a,\beta)\in\bE\times \bR$, such that $ f(x)\leq  \ip{a}{x}+\beta=:g(x)$ for all $x\in \dom f$.
Now pick $\bar x\in \dom f$. Then $(\bar x, g(\bar x) +1)\in \epi f$, and it hence suffices to show that $(\bar x, g(\bar x) +1)\not\in\clco(\gph F)$. Assume, by contradiction, that  $(\bar x, g(\bar x) +1)\in\clco(\gph F)$. Then with $\kappa:=\dim \bE\times \R$, by Carath\'eodory's theorem \cite[Theorem 17.1]{Roc 70}, for $i=1,\dots, \kappa+1$ there exist 
sequences $\{t_{i,k}\geq 0\}_{k\in\bN}$ and $\{x_i^k\in \dom f\}_{k\in\bN}$ such that  $\sum_{i=1}^{\kappa+1}t_{i,k}=1$ and 
$ \sum_{i=1}^{\kappa+1} t_{i,k}(x_i^k, f(x_i^k))\to_{k\to\infty} (\bar x, g(\bar x) +1)$. Consequently 
\[
g(\bar x)+1  =  \lim_{k\to\infty} \sum_{i=1}^{\kappa+1} t_{i,k}f(x_i^k)\leq \lim_{k\to\infty} \sum_{i=1}^{\kappa+1} t_{i,k} g(x_i^k)=\lim_{k\to\infty} \ip{a}{\sum_{i=1}^{\kappa+1} t_{i,k} x_i^k}+\beta=g(\bar x),
\]
which is the desired contradiction and thus concludes the proof.
\end{proof}

\noindent
The questions as to what can be said when $f$ in the above result is only proper and convex, but not necessarily closed is answered as the 
opening to Section \ref{sec:Suff}.

To start our analysis of the vector-valued case we now formally introduce the notion of {\em $K$-minorants and -majorants}, respectively.

\begin{definition}[$K$-minorants/-majorants]\label{def:KMin}
Let  $K\subset \bE_1 $ be a cone, and let  $F\colon\bE_1\to\bE_2^\bullet$ be proper.  A function $G:\bE_1\to\bE_2$ is  said to be:

\begin{itemize}
\item  a {\em $K$-majorant} of $F$ on $S\subset\dom F$ if 
\[
G(x)-F(x)\in K \quad \forall x\in S,
\]

\item a {\em $K$-minorant} of $F$  on $S\subset \dom F$ if 
\[
F(x)-G(x)\in K \quad\forall x\in S.
\]
For $S=\dom F$, we say that $G$ is a $K$-minorant of $F$.

\end{itemize}

\end{definition}

\noindent
Naturally, in view of the scalar case from Theorem \ref{p:epi=conv-gph-1Dcase}, we are mainly interested in the case where $G$ is an affine function.

For a function $F:\bE_1\to\bE_2^\bullet$,  we record that, for a pointed, closed, convex cone $K$ such that $\Epi{K}{F}=\clco(\gph F)$, there  cannot exist both an affine $K$-majorant on the domain and an affine $K$-minorant of $F$.

\begin{proposition}
\label{p:minmax} 
Let $\{0\}\subsetneq K\subset\bE_2$ be a  closed, convex,  pointed cone and let $F\colon\bE_1\to\bE_2^\bullet$ be proper. If $\Epi{K}{F}=\clco(\gph F)$, then $F$ cannot have both an affine $K$-majorant on its domain and an affine $K$-minorant.
\end{proposition}
\begin{proof}
Assume, by contradiction, that there  exist $T, L:\bE_1\to\bE_2$ be linear  and $u, w\in\bE_2$ be such that
\[
T(x)+u\geqslant_K F(x)\geqslant_K L(x)+w,\quad \forall x\in\dom F.
\]
We thus find that
$
\gph F\subset\set{(x,y)\in\dom F\times\bE_2}{T(x)-y\in -u+K,\; y-L(x)\in w+K}:=C.
$
As $K$ is closed and convex, so is $C$, and  we hence deduce that $\clco(\gph F)\subset C$. Now pick $x\in\dom F$ and $v\in K\setminus\{0\}$. Then 
$(x, F(x)+tv)\in\Epi{K}{F}$ for all $t>0$. Since  $\Epi{K}{F}=\clco(\gph F)\subset C$,  it follows that
\[
T(x)-F(x)-tv\in -u+K\AND F(x)+tv-L(x)\in w+K\quad \forall t>0.
\]
Dividing by $t$ and letting $t\to+\infty$ we get
$
 v\in K\cap(-K)=\{0\},
 $
which contradicts the choice of $v$, and therefore  proves the statement.
\end{proof}

\noindent
It is well known that a proper, convex function  posseses  an affine ($\R_+$-)minorant \cite[Theorem 9.20]{BaC 17}.  In the vector-valued setting, we can fall back on this result to get affine $K$-minorants for proper $K$-convex functions when $K$ is a particular  polyhedral cone.

\begin{proposition}\label{prop:ExAffPoly}

Let $K=\set{x\in\bE_2}{\ip{b_i}{x}\geq 0\quad\forall i=1,\dots,m}$ with   $b_1,\dots,b_m$  linearly independent. 
If $F:\bE_1\to\bE_2^\bullet$  is $K$-convex and  proper,   then $F$ has an affine $K$-minorant.
\end{proposition}
\begin{proof} It holds that $-K^\circ=\cone\{b_1,\dots,b_m\}$, see e.g.  \cite[Lemma 6.45]{RoW 98}. This cone is pointed by linear independence of $\{b_1,\dots,b_m\}$, hence $-K$ has nonempty interior, see   \cite[Exercise 6.22]{RoW 98}.
 Now, in view of Proposition \ref{p:lsc1} a), for all $i=1,\dots,m$, the functions $\ip{b_i}{F}:\bE_1\to\rp$ are proper, convex and hence, see e.g. \cite[Theorem 9.20]{BaC 17},  there exist $(c_i,\delta_i)\in \bE_1\times \R\;(i=1,\dots,m)$  such that
\begin{equation}\label{eq:AffMinAux}
\ip{b_i}{F(x)}\geq \ip{c_i}{x}+\delta_i\quad \forall x\in\bE_1,i=1,\dots,m.
\end{equation}
Now, let $A:\bE_1\to\bE_2$ be linear  such that $A(b_i)=c_i$ for all $i=1,\dots, m$ and let $w\in-\inter K.$ Then $\ip{w}{b_i}<0$ for all $i=1,\dots,m$, cf. \cite[Exercise 6.22]{RoW 98}. By positive homogeneity (and since $\inter K$ is a pre-cone), there hence exists $\bar w\in -\inter K$ with  $\ip{\bar w}{b_i}<\delta_i$ for all $i=1,\dots,m$.   Finally, with $\bar L:=A^*$ it hence follows
\[
\ip{b_i}{F(x)}\geq \ip{c_i}{x}+\delta_i \geq \ip{b_i}{\bar L(x)}+\ip{\bar w}{b_i}\quad\forall x\in\bE_1,i=1,\dots,m.
\]
Therefore, for all $x\in\dom F$, we have $F(x)-(\bar L(x)+\bar w)\in K$, and $x\mapsto \bar L(x)+\bar w$ is the desired affine $K$-minorant.
\end{proof}

\subsection{Sufficient conditions}\label{sec:Suff}

\noindent
In this subsection we are primarily concerned with sufficient conditions. We start with some considerations in the scalar case.

\begin{lemma}\label{lem:rint} Let  $f\in\Gamma(\bE)$.  Then  the following hold:
\begin{itemize}
\item[a)] $\clco(\gph \cl f)=\clco(\gph f_{|{\ri(\dom f)}})$. 
\item[b)] $\phi:\bE\to \R$ is an affine majorant of $\cl f$  on $\dom (\cl f)$ if and only if $\phi$ is  affine majorant of $f$ on $\ri(\dom f)$.
\item[c)] If $f$ is $(\dom f)$-closed (hence $f\in \Gamma_0(\dom f)$), then   $\phi:\bE\to \R$ is an affine majorant of $\cl f$  on $\dom (\cl f)$ if and only if $\phi$ is  affine majorant of $f$ on $\dom f$.
\end{itemize}

\end{lemma}
\begin{proof}a)  As $f(x)=\cl f(x)$ for all $x\in\ri(\dom f)$ (\cite[Theorem 7.4]{Roc 70}), we  have  $\gph f_{|{\ri(\dom f)}}\subset \gph \cl f$, and hence $\clco(\gph f_{|{\ri(\dom f)}})\subset\clco( \gph f)$. To prove the converse inclusion let $(x,\cl f(x))\in \gph \cl f$. Invoking \cite[Theorem 7.5]{Roc 70} (and  \cite[Theorem 6.1]{Roc 70}), we find a sequence $\{x_k\in \ri(\dom f)\}\to x$ with $f(x_k)\to \cl f(x)$. Therefore, $\gph \cl f\subset \cl(\gph f_{|{\ri(\dom f)}})\subset \clco(\gph f_{|{\ri(\dom f)}})$, and hence, the desired inclusion follows  by applying the $\clco$-operator on both sides.
\smallskip

\noindent
b) If $\phi$ is an affine majorant of $\cl f$ on $\dom(\cl f)$, then $\phi$ is an affine majorant of $\cl f$ on $\ri(\dom f)\subset\dom(\cl f)$, and hence  an affine majorant of $f$ on  $\ri(\dom f)$, since $f$ and $\cl f$ coincide on $\ri(\dom f)$. In turn, if $\phi$ is an affine majorant of $f$ on $\ri(\dom f)$, then for all $x\in\dom (\cl f)$, since $\ri(\dom (\cl f))=\ri(\dom f)$ (see \cite[Corollary 7.4.1]{Roc 70}), by  \cite[Theorem 7.5]{Roc 70} (and  \cite[Theorem 6.1]{Roc 70}), there exists $\{x_k\in\ri(\dom (\cl f))\}\to x$ with $\lim_k f(x_k) = \cl f(x_k)$. However $\phi(x_k)\geqslant f(x_k)$ and $\phi$ is continuous so $\phi(x)\geqslant \cl f(x)$, thus $\phi$ is an affine majorant of $\cl f$ on $\dom\cl f$.
\smallskip

\noindent
c) By Lemma \ref{l:f(x)=clf(x)-iff-f-in-Gamma_0(D)}, $f(x)=\cl f(x)$ for all $x\in \dom f$. Therefore, if $\phi(x)\geq \cl f(x)$ for all $x\in \dom (\cl f)\supset \dom f$, then $\phi(x)\geq f(x)$ for all $x\in \dom f$. In turn, if $\phi$ is an affine minorant of $f$ on $\dom f\supset \ri(\dom f)$, then b) shows that $\phi$ is an affine minorant of $\cl f$ on $\dom (\cl f)$.

\end{proof}

\noindent
We record some immediate consequences of the foregoing result.

\begin{corollary}
\label{cor:Gamma_0,dom-convgphf=clepif}
Let $f\in\Gamma(\bE)$. Then the following are equivalent:
\begin{itemize}
\item[i)]  $\clco(\gph f_{|{\ri(\dom f)}}) = \cl(\epi f)$;
\item[ii)] $f$ has no affine majorant on $\ri(\dom f)$;
\item[iii)] $\{0\}\times\bR_+ \subset [\clco(\gph f_{|{\ri(\dom f)}})]^\infty$.
\end{itemize}
\begin{proof}  Observe that $\epi(\cl  f) =\cl (\epi f)$,  hence  by Lemma \ref{lem:rint} a) we have 
\begin{equation}\label{eq:i)}
i) \iff  \clco(\gph \cl f)=\epi(\cl f).
\end{equation}
\noindent
'i)$\Leftrightarrow$ii)': By Lemma \ref{lem:rint} b), we have that ii) is equivalent to saying  that $\cl f$ has no affine minorant on its domain. Therefore, the desired equivalence follows with \eqref{eq:i)} from Proposition \ref{p:epi=conv-gph-1Dcase} applied to $\cl f\in \Gamma_0(\bE)$.
\smallskip

\noindent
'i)$\Leftrightarrow$iii)': Apply Theorem \ref{t:rec_cnd} to $\cl f$ and use \eqref{eq:i)}.
\end{proof}

\end{corollary}

\begin{corollary}
\label{c:Gamma_0,dom-convgphf=clepif}
Let $f\in\Gamma_0(\dom f)$. Then the following are equivalent:
\begin{itemize}
\item[i)]$\clco(\gph f) = \cl(\epi f)$;
\item[ii)]  $f$ has no affine majorants on its domain;
\item[iii)]$\{0\}\times\bR_+ \subset [\clco(\gph f)]^\infty$.
\end{itemize}
\end{corollary}

\begin{proof} We observe that $f\in\Gamma(\bE)$ (by definition of $\Gamma_0(\dom f)$),  and that  $f(x)=\cl f(x)$  for all $x\in \dom f$,  by Lemma \ref{l:f(x)=clf(x)-iff-f-in-Gamma_0(D)}.  In addition, by Lemma \ref{lem:rint}, we have  $\clco(\gph f|_{\ri(\dom f)}) = \clco(\gph\cl f)$. Thus, we have 
 \[
 \clco(\gph f)\subset\clco(\gph \cl f)=\clco(\gph f|_{\ri(\dom f)})\subset\clco(\gph f),
 \] 
 and hence $\clco(\gph f)=\clco(\gph f|_{\ri(\dom f)})$.  Consequently
 \[
 i) \iff  \clco(\gph f_{|{\ri(\dom f)}}) = \cl(\epi f),
 \]
and
\[
iii) \iff \{0\}\times\bR_+ \subset [\clco(\gph f_{|{\ri(\dom f)}})]^\infty.
\]
Moreover, with Lemma \ref{lem:rint} we find that
\[ 
ii) \iff f\;\text{has no affine majorant on}\; \ri(\dom f).
\] 
Therefore, the claimed equivalences follow from Corollary \ref{cor:Gamma_0,dom-convgphf=clepif}.

\end{proof}

\noindent
We now establish sufficient conditions in the vector-valued case, building on the results provided above. We start with the most general result, and then successively tighten the assumptions to obtain (weaker but) more handy conditions.

\begin{lemma}
\label{l:Kepi=convgph-bis}
Let $ K\subset\bE_2$ be a (nontrivial) closed, convex cone  with $K\subset -K^\circ$, and let $F\colon\bE_1\to\bE_2^\bullet$ be proper, $K$-closed and $K$-convex. Assume  that the following hold:
\begin{itemize}
\item[i)] There is a nonempty set $L\subset K\cap\rge F$ such that $\overline{\cone}L=K$;
\item[ii)] For every  $u\in L$, there exists a convex set $C^u\subset\dom F$ such that $
F(C^u) \subset \bR_+u$;
\item[iii)]  For all $u\in L$ we have  $f_u := \langle u,F\rangle + \delta_{C^u}\in \Gamma_0(C^u)$ and  it has no affine majorant on its domain.
\end{itemize}
Then $\Epi{K}{F}=\clco(\gph F)$. In particular, $K=\hat{K}_F$.
\end{lemma}

\begin{proof}
Let $u\in L \setminus\lbrace 0\rbrace$, and let us prove that $(0,u)\in [\clco(\gph F)]^\infty$. As the latter is a cone, we may assume w.l.o.g that $\|u\|=1$. By iii),  $f_u\in\Gamma_0(C^u)$, hence  Corollary \ref{c:Gamma_0,dom-convgphf=clepif} yields that $\lbrace 0\rbrace\times\bR_+ \subset [\clco(\gph f_u)]^\infty$. 
As $\co(\gph f_u)$ is convex, we  hence have  $(0,t)+\co(\gph f_u)\subset\clco(\gph f_u)$ for all $t\geq 0$. 

Now, let  $(x,r)\in\co(\gph f_u)$.  Hence, with $\kappa:=\dim \bE\times \R$, there there exist  convex combinations   $x=\sum_{i=1}^{\kappa+1} \alpha_ix^i$ and $r=\sum_{u=1}^{\kappa+1}\alpha_i \ip{u}{F}(x^i)$ with $x^i\in C^u\;(i=1,\dots,\kappa+1)$.   By ii),   there exists $\gamma\geqslant 0$ such that
$\sum_{i=1}^{\kappa+1} \alpha^i F(x^i) = \gamma u,$ and thus $\gamma=\ip{u}{\gamma u}=r$. Consequently, $(x,ru)\in\co(\gph F)$.

Moreover, we have  $(x,r+t)\in\clco(\gph f_u)$ for all $t\geq 0$, hence 
 \[
 x=\lim_{n\to\infty}\sum_{i=1}^{\kappa+1} \alpha^{i,n}_tx^{i,n}_t \AND r+t=\lim_{n\to\infty}\sum_{i=1}^{\kappa+1}\alpha^{i,n}_t \ip{u}{F}(x^{i,n}_t)
\]
for certain $x^{i,n}_t\in C^u$ and $\alpha^{i,n}_t\geq 0\;(i=1,\dots,\kappa+1)$ with $\sum_{i=1}^{\kappa+1} \alpha^{i,n}=1$ for all $n\in \bN$.
However, by ii),  there exist $\gamma_t^n\geqslant 0\;(n\in\bN)$ such that
\[
\sum_{i=1}^{\kappa+1} \alpha^{i,n}_t F(x^{i,n}_t)=\gamma_t^n u\quad \forall n\in \bN.
\]
Thus
\[
\gamma_t^n=\ip{u}{\sum_{i=1}^{\kappa+1} \alpha^{i,n}_t F(x^{i,n}_t)}= \sum_{i=1}^{\kappa+1}\alpha^{i,n}_t \ip{u}{F}(x^{i,n}_t) \to r+t .
\]
Thus,   $(x,ru+tu)=(x,(r+t)u)\in\clco(\gph F)$  for all $t\geqslant 0$. Now, for all $(\bar x,\bar y)\in\co(\gph F)$, and for all $s\geqslant 0$, we have 
\[
(\bar x,\bar y+su) = \lim_{\varepsilon\downarrow 0} (1-\varepsilon) (\bar x,\bar y)+\varepsilon (x,ru+(s/\varepsilon) u) .
\]
However, by what was argued above, $(x,ru+(s/\varepsilon) u)\in\clco(\gph F)$ for all $\varepsilon>0$,  and we picked $(\bar x,\bar y)\in\co(\gph F)\subset\clco(\gph F)$. By convexity and closedness , it follows that  
\[
(\bar x,\bar y+su)\in\clco(\gph F)\quad \forall (\bar x,\bar y)\in\co(\gph F),\; s\geq 0.
\]
Thus $(0,u)\in [\co(\gph F)]^\infty$ for all $u\in L\setminus \{0\}$ as desired.
For  $u=0$ this is trivially true, hence, altogether we find that $\{0\}\times L\subset [\clco(\gph F)]^\infty$.  But as $[\clco(\gph F)]^\infty$ is a closed convex cone, we have then
$
\overline{\cone}(\{0\}\times L)=\{0\}\times K \subset[\clco(\gph F)]^\infty ,
$ cf. i).  Thus, by Theorem \ref{t:rec_cnd}, we have $\Epi{K}{F}=\clco(\gph F)$.
\end{proof}

\noindent
We record an immediate consequence. 

\begin{proposition}
\label{p:Kepi=convgph-2} Let  $K\neq\lbrace 0\rbrace$ be  a closed, convex cone such that $K\subset -K^\circ$ and such that $F\colon\bE_1\to\bE_2^\bullet$ is proper,  $K$-convex and $K$-closed. Assume  that $K=\cone (b_1,\dots,b_N)$ for $b_1,\dots,b_n\in\rge F$, and that, for any $i=1,\dots,N$, there exists a nonempty convex set $C^{i}\subset F^{-1}(\bR_+b_i)$ such that,  for all $i=1,\dots,N$, the function 
$
f_{b_i}:=\langle b_i,F\rangle +\delta_{C^{b_i}}
$
is $C^{b_i}$-closed and  does not have any affine majorant on $C^{b_i}$.
Then  $\Epi{K}{F}=\clco(\gph F)$.
\end{proposition}
\begin{proof}
Apply Lemma \ref{l:Kepi=convgph-bis}, with $L=\{b_1,\dots,b_N\}$.
\end{proof}

\noindent
To wrap up this section we want to provide a simplified version of Lemma \ref{l:Kepi=convgph-bis} with  more restrictive, but less arduous assumptions. To this end, we need the following lemma.

\begin{lemma}
\label{l:relative-interior-of-K-when-K-in-polar-cone}
Let $K\subset\bE_2$ be a closed, convex cone  with $K\subset -K^\circ$, and let  $F\colon\bE_1\to\bE_2^\bullet$ be proper, $K$-convex and $K$-closed. Then   $\ip{u}{F}\in\Gamma_0(\bE_1)$ for all $u\in\ri K$.
\end{lemma}
\begin{proof}  We observe from   \cite[Corollary 7.4(ii)]{Pen 99} that   $\ip{u}{F}\in\Gamma_0(\bE_1)$ for all $u\in\ri(-K^\circ)$. Thus if  $\ri K\subset\ri(-K^\circ)$ there is nothing to prove. 

Hence, we only need to consider the case  $\ri K\nsubseteq\ri(-K^\circ)$. Since we assume that $K\subset -K^\circ$, by the definition of the relative topology, this can only hold, if    $\aff K \varsubsetneq\aff(K^\circ)$. We note that both of these sets contain $0$, and hence are subspaces of $\bE_2$. In particular,  the orthogonal projection  $p:\bE_2\to \bE_2$   onto $\aff K$ (which is ordered by $K$) is a linear self-adjoint operator. We define $G:\bE_1\to (\aff K)^\bullet$ by 
\[
G(x):=\begin{cases} p(F(x)),& x\in \dom F,\\
+\infty_\bullet, & \text{else.}
\end{cases}
\]
Note that for all $\alpha\in (0,1)$ and  $x,y\in\dom F$,  we have
$
\alpha F(x)+(1-\alpha) F(y) - F(\alpha x+(1-\alpha)y) \in K.
$
Hence, as $K\subset \aff K$ and by linearity of $p$, for all $\alpha\in (0,1)$ and $x,y\in\dom F=\dom G$, we have
\[
\alpha F(x)+(1-\alpha) F(y) - F(\alpha x+(1-\alpha)y) = \alpha G(x)+(1-\alpha) G(y) - G(\alpha x+(1-\alpha)y).
\]
Therefore, $G$ is $K$-convex. Moreover, if we denote $D:=\dom G$ and  $H:=F-G:D\to \aff(K)$, then 
\[
\alpha H(x)+(1-\alpha) H(y) - H(\alpha x+(1-\alpha)y) = 0\quad \forall x,y\in D,\;\alpha\in(0,1).
\]
Hence, $H$ is $\{0\}$-convex, and by Proposition \ref{p:0-convexity-and-affine-functions}, there exists an affine function $\hat H:\bE_1\to \aff K$ such that $\hat H_{|D}=H$.

Now, let $\{(x_k,z_k)\in\Epi{K}{G}]\}\to (x,z)\in \bE_1\times\aff K$. Then, for all $k\in \bN$,  $x_k\in\dom G$, and there exists $v_k\in K$ such that 
\[
z_k = G(x_k)+v_k = F(x_k)-H(x_k)+v_n=F(x_k)-\hat{H}(x_k)+v_k.
\]
As $\hat{H}:\bE_1\to \aff K$ is affine, it is continuous, so $\hat{H}(x_k)\to \hat H(x)$,  and thus  $F(x_k)+v_k\to z+\hat{H}(x)$. Therefore $\{(x_k,F(x_k)+v_k)\in\Epi{K}{F}\}\to(x,z+\hat{H}(x))$. As $F$ is $K$-closed, we have $(x,z+\hat{H}(x))\in\Epi{K}{F}$, thus $x\in\dom F=\dom G$, $\hat{H}(x)=H(x)$, and $z+H(x)-F(x)\in K$, so $z-G(x)\in K$ and $(x,z)\in\Epi{K}{G}$. This proves that $G$ is $K$-closed.

Let  $K'$ be the dual cone of  $K$ in $\aff K$. As $K\subset -K^\circ$ by assumption,  we  consequently   obtain 
$
K\subset -K^\circ\cap\aff K= K'\subset \aff K.
$
Hence, $\ri K \subset\ri K'$. Moreover, as $G:\bE_1\to (\aff K)^\bullet$ is proper, $K$-closed and $K$-convex, by \cite[Corollary 7.4(ii)]{Pen 99}, we have $\ip{u}{G}\in\Gamma_0(\bE_1)$  for all $u\in\ri K'$. But for  any $u\in\ri K\subset \aff K$, as $p$ is self-adjoint, we have  $\ip{u}{G}=\ip{u}{p(F)}=\ip{u}{F}$. Thus, for any $u\in\ri K$ we have $\ip{u}{F}\in\Gamma_0(\bE_1)$.
\end{proof}

\begin{proposition}
\label{p:Kepi=convgph-1}
Let  $K\subset\bE_2$ be a proper, closed, convex cone such that $K\subset -K^\circ$ and let  $F\colon\bE_1\to\bE_2^\bullet$ be proper,  $K$-convex and $K$-closed with $\ri K\subset\rge F$. Moreover, assume that,  for any $u\in\ri K$, there exists a nonempty convex set $C^u\subset F^{-1}(\bR_+u)$ such that  
$
f_u:= \langle u,F\rangle + \delta_{C^u}
$
does not have any affine majorant on its domain. 
Then, $\Epi{K}{F}=\clco(\gph F)$.
\end{proposition}

\begin{proof}
By Lemma \ref{l:relative-interior-of-K-when-K-in-polar-cone}, for all $u\in\ri K$, we have $\ip{u}{F}\in\Gamma_0(\bE_1)$, hence $f_u\in\Gamma_0(C^u)$. Applying  Lemma \ref{l:Kepi=convgph-bis}  with $L=\ri K $ yields the desired result.
\end{proof}

\subsection{Examples}\label{sec:Ex}

\noindent
In this section we put our findings from the previous sections to the test on various examples of $K$-convex functions. Throughout, we equip
the matrix space $\R^{n\times m}$ with the {\em Frobenius} inner product $\ip{\cdot}{\cdot}:\R^{n\times m}\times \R^{n\times m}\to \R$, $\ip{X}{Y}=\tr(X^TY)$. In particular, on the space of symmetric matrices $\bS^n$, the transposition is superfluous.

\subsubsection{\boldmath{$F:X\mapsto \frac{1}{2}XX^T$}}

We consider the function 
\begin{equation}\label{eq:GMF}
F:\bR^{n\times m}\to\bS^n, \quad F(X)=\frac{1}{2}XX^T.
\end{equation}

\noindent
It plays a central role in study of the {\em matrix-fractional} \cite{BuH 15, BGH 17, BGH 18} and {\em variational Gram functions} \cite{BHN 21, JFX 17}.

\begin{proposition}\label{prop:GMF} Let $F$ be given by \eqref{eq:GMF}. Then the following hold:
\begin{itemize}
\item[a)] $\hat K_F=K_F=\bS^n_+=\co (\rge F)$.  In case where $m=n$, the convex hull is superfluous.
\item[b)] $F$ is $\bS^n_+$-closed and -convex.
\item[c)]$\clco (\gph F)=\Epi{\bS^n_+}{F}$.
\end{itemize}

\end{proposition}
\begin{proof} a) We know from \cite[Lemma 8]{BHN 21} that $K_F=\bS^n_+$. But as $F$ is continuous and $K_F$ is closed, we have that $F$ is $K_F$-closed, and hence $K_F=\hat K_F$, which shows the first identity. 
For the third, observe that, clearly  $\bS^n_+\supset\co (\rge F)$. On the other hand for $V\in\bS^n_+$, there exists $L\in \R^{n\times n}$  such that  $\frac{1}{2}LL^T=V$. This already shows that $\bS^n_+= \rge F$ if $m=n$. If not, we denote the columns of $L$ by $\ell_1,\dots,\ell_n$ and  set $ x_i:=\sqrt{n}\ell_i$ for all $i=1,\dots,n$, and $X_i:=[x_i,0, \dots,0]\in \R^{n\times m}$. Then 
\[
V=\frac{1}{2}\sum_{i=1}^n \ell_i \ell_i^T=\frac{1}{2}\sum_{i=1}^n\left(\frac{x_i}{\sqrt n}\right)\left(\frac{x_i}{\sqrt n}\right)^T=\sum_{i=1}^n \frac{1}{n}F(X_i)\in \co(\rge F),
\]
which gives the desired inclusion.
\smallskip

\noindent
b) Follows from a).
\smallskip

\noindent
c) We prove that $\{0\}\times\bS_+^n\subset \co(\gph F)$ which then gives the desired result via b) and Corollary \ref{cor:Suff}. To this end, let $(0,V)\in\{0\}\times\bS_+^n$. Hence, by a),  there exist $\alpha_1,\dots,\alpha_r\geqslant 0$ and $X_1,\dots X_r \in\bR^{n\times m}$ such that $\sum_{i=1}^r \alpha_i = 1$ and $V=\sum_{i=1}^r\alpha_i F(X_i)$. However, $F(-X_i)=F(X_i)$. Hence, we also have 
$
V = \sum_{i=1}^r \frac{\alpha_i}{2}F(X_i) + \frac{\alpha_i}{2}F(-X_i) .
$
As $\sum_{i=1}^r \frac{\alpha_i}{2}X_i + \frac{\alpha_i}{2}(-X_i)= 0$, we then have $(0,V)\in\co(\gph F)$.
\end{proof}

\subsubsection{The squared matrix mapping}

We consider the map 
\begin{equation}\label{eq:Sq}
F:\bS^n\to \bS^n,\quad F(X)=X^2. 
\end{equation}

\begin{proposition} Let $F$ be given by \eqref{eq:Sq}. Then the following hold:
\begin{itemize}
\item[a)]  $\hat K_F=K_F=\bS^n_+=\rge F$.
\item[b)] $F$ is $\bS^n_+$-closed and -convex.
\item[c)]$\clco (\gph F)=\Epi{\bS^n_+}{F}$.
\end{itemize}
\begin{proof} a) Using Proposition \ref{prop:KF} b) we know that $-K_F^\circ=\set{V\in\bS^n}{\ip{V}{F}\;\text{convex}}$.
Now for any $V\in\bS^n$, we have  $\nabla \ip{V}{F}(X)=VX+XV$. Therefore, for $X,Y\in \bS^n$, we find that
\[
\ip{V}{F}(X)-\ip{V}{F}(Y)+\ip{\nabla\ip{V}{F}(X)}{Y-X}=-\tr\left((X-Y)V(X-Y)\right).
\]
For $\ip{V}{F}$ to be convex, by the gradient inequality, it is therefore  necessary and sufficient that 
\[
\tr\left((X-Y)V(X-Y)\right)\geq 0\quad \forall X,Y\in\bS^n,
\]
which is equivalent to saying that $V\succeq 0$. Therefore $-K_F^\circ =\bS^n_+$, and by bipolarity, we obtain $K_F=\bS^n_+$. Since $F$ is continuous, we have $K_F=\hat K_F$, and the fact that $\rge F=\bS^n_+$ is obvious.
\smallskip

\noindent
b) Follows from a).
\smallskip

\noindent
c) Use the same reasoning as in the proof of Proposition \ref{prop:GMF} c).
\end{proof}

\end{proposition}

\subsubsection{The inverse matrix mapping} 

\noindent
We consider  the  map 
\begin{equation}\label{eq:Inv}
F:\bS^n_{++}\to\bS^n,\quad F(X)=X^{-1}.
\end{equation}

\begin{proposition} Let $F$ be given by \eqref{eq:Inv}. Then the following hold:
\begin{itemize}
\item[a)] $\hat K_F=K_F=\bS^n_+$.
\item[b)] $F$ is $\bS_+^n$-convex and -closed. 
\item[c)]  $\clco (\gph F)=\Epi{\bS^n_+}{F}$.
\end{itemize}

\end{proposition}
\begin{proof} a)  By Proposition \ref{prop:KF} b), we know that $-K_F^\circ=\set{V\in\bS^n}{\ip{V}{F}\;\text{convex}}$. Now let $V\in\bS^n$ and observe that $\nabla\ip{V}{F}(X)=-X^{-1}VX^{-1}$ for all $X\succ 0$. Therefore, for all $X,Y\succ 0$
\begin{eqnarray*}
\lefteqn{\ip{V}{F}(X)-\ip{V}{F}(Y)+\ip{\nabla\ip{V}{F}(X)}{Y-X}}\\
 & = & \tr(VX^{-1})-\tr(VY^{-1})-\tr(X^{-1}VX^{-1}(Y-X))\\
 &= & -\tr(V(Y^{-1}+X^{-1}YX^{-1}))
\end{eqnarray*}
For $\ip{V}{F}$ to be convex, by the gradient inequality, it is therefore  necessary and sufficient that 
\[
\tr(V(Y^{-1}+X^{-1}YX^{-1})\geq 0\quad \forall X,Y\succ 0.
\]
Since the trace of a product of two positive semidefinite matrices is nonnegative, it follows that $V\succeq 0$ will be sufficient. On the other hand, let $V\not\succeq 0$ and let $V=\sum_{i=1}^n \lambda_i q_iq_i^T$ be the spectral decomposition with $q_1,\dots,q_n$ an orthonormal basis of $\R^n$ and $\lambda_n<0$. Now define 
\[
X_t:=\left(\sum_{i=1}^{n-1} q_iq_i^T +tq_nq_n^T+\frac{1}{t}I\right)^{-1}\succ 0\quad \forall t>0.
\]
Put $Y_t:=X_t\succ 0$. Then, since  $q_1,\dots,q_n$ are  orthonormal,  we have
\begin{eqnarray*}
\tr(V(Y_t^{-1}+X_t^{-1}Y_tX_t^{-1})& = & 2\tr\left(V\left(\sum_{i=1}^{n-1}  q_iq_i^T +tq_nq_n^T+\frac{1}{t}I\right)\right)\\
&= & 2\left[\tr\left(\sum_{i=1}^{n-1} \lambda_iq_iq_i^T\right)+\tr(t\lambda_nq_nq_n^T)+\tr\left(\frac{1}{t}V\right)\right]\\
&= & 2\left[\sum_{i=1}^{n-1} \lambda_i +t\lambda_n+\frac{\tr(V)}{t}\right]\\
& \overset{t\to\infty}{\to}& -\infty.
\end{eqnarray*}
Therefore $-K_F^\circ =\bS^n_+$, and by bipolarity, we obtain $K_F=\bS^n_+$. In view of Corollary \ref{prop:ConeEq} it hence  suffices to show that $F$ is $\bS^n_+$-closed: To this end, let $\{(X_k,Y_k)\in \Epi{\bS^n_+}{F}\}\to (X,Y)$. In particular, $Y_k=X_k^{-1}+S_k$ for $S_k\succeq 0$.  It hence is enough to show that $X\succ 0$. If this were not the case, then  the smallest eigenvalue $\lambda_k$ of $X_k$ goes to zero, and hence $1/\lambda_k$, the largest eigenvalue of $X_k^{-1}$, goes to $+\infty$. This cannot be compensated for by $S_k\succeq 0$, which then contradicts that $\{X_k^{-1}+S_k\}$ is convergent, hence bounded. 
\smallskip

\noindent
b) Follows from a).
\smallskip

\noindent
c)  Firs note that $\rge F=\bS^n_{++}=\ri \bS^n_+$ and that $\bS^n_+$ is self-dual, i.e. $\bS^n_+=-(\bS^n_+)^\circ$.  Moreover, for every $U\in \ri \bS^n_+$, we have that $C^U=F^{-1}(\R_+U)=\R_{++}U^{-1}$ is convex and nonempty. The desired statement will follow  from
  Proposition \ref{p:Kepi=convgph-1} once we show that $\ip{U}{F}$ has no affine majorant on $C^U$. To this end, let $V_t=tU^{-1}\in C^U$  for  $t>0$.  Then $\ip{U}{F}(V_t)=\frac{\|U\|^2}{t}$. Since $0<t\mapsto 1/t$ has no affine majorant on $\R_{++}$, then $\ip{U}{F}$ cannot have an affine majorant on $C^U$.
\end{proof}

\subsubsection{Entry-wise convex functions} 

It is well known \cite{BHN 21} that a component-wise convex  function $\bE_1\to \R^n$ is $\R_{+}$-convex.
This can be slightly  generalized. 

\begin{proposition}
Let $\{b_1,\dots,b_n\}\subset \bE_2$   and let $f_i:\bE_1\to\R$ be convex for all $i=1,\dots,n$. Define $F:\bE_1\to \bE_2$ by
$
F(x)=\sum_{i=1}^nf_i(x)b_i
$
and let $K:=\cone \{b_1,\dots,b_n\}.$ Then the following hold:
\begin{itemize}
\item[a)] $F$ is $K$-convex and $K$-closed.
\item[b)] We have $K\subset -K^\circ$ if and only if $\ip{b_i}{b_j}\geq 0$ for all $i,j=1,\dots,n$. 
\item[c)] Suppose that, for all $i=1,\dots,n$, we have  $C_i:=\bigcap_{i\neq j}\argmin f_j\neq \emptyset$ and that $f_i$ has no affine majorant on $C_i$. Then $\Epi{K}{F}=\clco(\gph F)$.
\end{itemize}

\end{proposition}
\begin{proof} 
a) Observe that $-K^\circ=\set{y}{\ip{b_i}{y}\geq 0\;(i=1,\dots,n)}$. Therefore  for all $z\in -K^\circ$ we have 
$\ip{z}{F}=\sum_{i=1}^n \ip{b_i}{z} f_i\in\Gamma_0(\bE_1).$ Hence  Propostion \ref{p:lsc1} yields the desired statement.
\smallskip

\noindent
b) Clear as $K^{\circ\circ}=K$.
\smallskip

\noindent
c)  By assumption,  $\ip{F}{b_i}+\delta_{C^i}\in\Gamma_0(C^i)$ and $C^i\subset F^{-1}(\bR_+ \{b_i\})$. The claim hence follows from  Proposition \ref{p:Kepi=convgph-2}.
 \end{proof}

\subsubsection{The spectral function}

The  {\em spectral function}  \cite{BHN 21, Lew 96, Pen 99} is the map $\lambda\colon\bS^n\to\bR^n, \;\lambda(A)=\left[\lambda_1(A),\dots\lambda_n(A)\right]^T$ where $\lambda_1(A)\geq\dots\geq\lambda_n(A)$ are the ordered eigenvalues of $A$ (with multiplicity). 
Define the cone
\begin{equation}\label{eq:SpecK}
K_n=\left\lbrace v\in\bR^n\mid \sum_{i=1}^k v_i \geq 0 , k=1,\dots,n-1,\sum_{i=1}^n v_i = 0 \right\rbrace.
\end{equation}

\noindent
The following result clarifies the convexity properties of $\lambda$ w.r.t $K_n$ and shows, based on  Proposition \ref{p:Kepi=convgph-2} and   Corollary  \ref{c:K-included-in-minus-polar-cone}, respectively, that the question whether  $\Epi{K_n}{\lambda}=\clco(\gph \lambda)$ depends on $n$.

\begin{proposition}[Spectral map]\label{prop:Spec} Let $K_n$ be given by \eqref{eq:SpecK}. Then the following hold:
\begin{itemize}
\item[a)]  $K_n$ is closed, convex and pointed with $K_n^\circ=\set{w\in\R^n}{w_1\geq\dots\geq w_n}$. 
\item[b)] $\lambda$ is $K_n$-convex and $K_n$-closed.
\item[c)] The following are equivalent:
\begin{itemize}
\item[i)] $K_n\subset -K_n^\circ$;
\item[ii)] $n\leq 2$;
\item[iii)]  $\Epi{K_n}{\lambda}=\clco(\gph \lambda)$.
\end{itemize}

 \end{itemize}

\end{proposition}

\begin{proof}  a) The properties of $K$ are straightforward. The formula for $K_n^\circ$ can be found in  e.g. \cite{BHN 21, Lew 96, Pen 99}.
\smallskip 

\noindent
b)  See e.g. \cite{BHN 21, Lew 96, Pen 99} for $K_n$-convexity. The $K_n$-closedness follows because $\lambda$ is continuous on $\bS^n$ and $K_n$ is closed. 
\smallskip

\noindent
c) Consider the following implications:
\begin{itemize} 
\item[]'i)$\Rightarrow$ii)': For $n>2$ we have $[0,\dots,0,1,-1]^T\in K_n\setminus (-K_n^\circ)$, see a). 
\item[] 'ii)$\Rightarrow$iii)':   For $n=1$ there's nothing to prove.  For $n=2$ set  $b_1:= [1;-1]^T$ so that $K_n=\cone \{b_1\}$  and  define $C^{b_1} := \left\{ \left[\begin{array}{cc}
\alpha & 0 \\
0 & -\alpha 
\end{array}\right] \mid \alpha\in\bR \right\} $ which is a subspace, hence nonempty and closed, and convex.  Then $C^{b_1}\subset \lambda^{-1}(\bR_+b_1)$ and we have
$
 \langle b_1,\lambda\rangle\left(\left[\begin{array}{cc}
\alpha & 0 \\
0 & -\alpha 
\end{array}\right]\right) = 2|\alpha|
$  for all  $\alpha\in \R$.
Therefore $\ip{b_1}{\lambda}+\delta_{C^{b_1}}\in \Gamma_0(\bS^n)$ and  has no affine majorant on its domain $C^{b_1}$. Therefore Proposition \ref{p:Kepi=convgph-2} yields the desired implication.
\end{itemize}

\smallskip

\noindent
'iii)$\Rightarrow$i)': Corollary  \ref{c:K-included-in-minus-polar-cone}.
\end{proof}

\section{Convex-convex composite functions} \label{sec:gF}

We start this section with the definition of $K$-increasing functions.

\begin{definition}[$K$-increasing functions]\label{def:KIncr} Let $K\subset \bE$ be a  cone. The  function  $g:\bE\to\bR\cup\{+\infty\}$ is said to be {\em $K$-increasing} if 
\[
y\geq_K x \quad \Longrightarrow\quad  g(y)\geq g(x) \quad \forall x,y\in \bE.
\]
\end{definition}

\noindent
It is well known and explored extensively in the literature \cite{BHN 21,Bot 10,Bot et al. 06,BGW 07} that, given $K\subset\bE$, a $K$-increasing function $g\in\Gamma(\bE_2)$ and a $K$-convex function $F:D\subset\bE_1\to\bE_2$,  the composition
\begin{equation}\label{eq:CC}
g\circ F:\bE_1\to\rp,\quad (g\circ F)(x):=\begin{cases} g(F(x)),& x\in D,\\ +\infty, & \text{else}
\end{cases}
\end{equation}
is convex (and proper if and only if $F(\dom F)\cap \dom  g\neq \emptyset$).   One of the questions we address in this section is the following: given $g\in\Gamma(\bE_2)$ and $F:D\subset \bE_1\to\bE_2$ such that    $g\circ F$ is convex, under which conditions does there exist a  (closed) cone $K$ such that $F$ is $K$-convex and $g$ is $K$-increasing?

\subsection{The horizon cone of a closed, proper, convex function}

For a proper  function  $f\colon\bE\to\rbar$,  its {\em horizon function} $f^\infty:\bE\to \rbar$ is defined  via  
 $\epi f^\infty = (\epi f)^\infty$.
The {\em horizon cone} $\hzn f$ of $f$ is the level set
$
\hzn f:=\set{x\in \bE}{f^\infty(x)\leq 0}.
$
For $f\in\Gamma_0(\bE)$ the horizon function and  horizon cone of $f$ coincide with the respective  {\em recession} objects \cite[Chapter 8]{Roc 70}.
We summarize some fundamental properties of the horizon cone of a closed, proper, convex function.

\begin{proposition}
\label{p:K-increaseness}
Let $g\in\Gamma_0(\bE)$. Then the following hold: 
\begin{itemize}
\item[a)] $g^\infty$ is closed, proper, convex and positively homogenous.
\item[b)] We have
\[
g^\infty(u)=\sup_{t>0} \frac{g(x+tu)-g(x)}{t}\quad \forall x\in \dom g.
\]
In particular, we have 
\[
u\in\hzn g \iff g(x+tu)\leqslant g(x)\quad \forall x\in \dom g, t>0.
\]
\item[c)] $\hzn g$ is a closed convex cone. 

\item[d)] $g$ is   $(-\hzn f)$-increasing.
\item[e)]  $K$ is a cone with respect to which $f$ is increasing if and only if $K\subset -\hzn f$.

\end{itemize}
\end{proposition}
\begin{proof} a),b) See \cite[Theorem 3.12]{RoW 98}.
\smallskip

\noindent
c)  From a) and the definition of $\hzn f$.
\smallskip

\noindent
d) See    \cite[Lemma 7]{BHN 21} or \cite[Corollary 3.1]{LLY 20}.
\smallskip

\noindent
e)  See \cite[Proposition 3.2]{LLY 20}.
\end{proof}

\noindent
The next example shows that the convexity in part b) is essential, which also shows that the horizon function is not the recession function  (see  \cite{LLY 20}) without convexity.

\begin{example}\label{ex:Noncvx}
Consider $f:\R\to \R$ given by 
\[
f(x)=\begin{cases} 1+x,& x< -1,\\
0, & x\in[-1,1],\\
1-x,& x>1.
\end{cases}
\]
Then $f$ is continuous (hence proper and lsc), but not convex, and it holds that $f^\infty(u) = -|u|$. 
Moreover, for any $u\in\bR$, $\sup_{x\in\bR,t>0} \frac{f(x+tu)-f(x)}{t} = |u|$. Thus $f^\infty(u)\neq\sup_{x\in\bR,t>0} \frac{f(x+tu)-f(x)}{t}$.
\end{example}

\subsection{The \boldmath{$K$}-increasing case}

The next proposition  characterizes the situation where  there exists a cone with respect  which $F:D\to \bE_2$ is convex and $g\in\Gamma_0(\bE_2)$ is increasing. At this, the cone $K_F$,  the smallest closed cone  with respect to which $F$ is convex, comes in to play, which  ties our study here to our  results from Section \ref{sec:Kcvx}.

\begin{proposition} \label{p:condition-F-Kconvex-g-Kincreasing}
Let $g\in\Gamma_0(\bE_2)$ and $F:D\to\bE_2^\bullet$ with $D\subset \bE_1$ convex such that $g\circ F$ is proper. Then the following statements are equivalent.
\begin{enumerate}
\item[i)] There exists a cone $K$ such that $g$ is $K$-increasing and $F$ is $K$-convex;
\item[ii)] $g$ is $K_F$-increasing;
\item[iii)] $K_F\subset -\hzn g$;
\item[iv)] $(\hzn g)^\circ \subset -K_F^\circ$.
\end{enumerate}
\end{proposition}

\begin{proof} We only (need to) show that i), ii) and iii) are equivalent. The equivalence of iii) and iv)  follows from (bi)polarization and the fact that both cones in play are closed and convex.
\smallskip

\noindent
i)$\Rightarrow$ii) : Let   $K\subset \bE_2$ such that $F$ is $K$-convex and $g$ is $K$-increasing. In particular, by Proposition \ref{p:K-increaseness} e), we have  $K\subset -\hzn g$,  and thus  $F$ is $(-\hzn g)$-convex and $g$ is $(-\hzn g)$-increasing, by Proposition \ref{p:K-increaseness} d). As $(-\hzn g)$ is closed and convex, see Proposition \ref{p:K-increaseness} c),   by definition of $K_F$ we have 
$K_F\subset -\hzn g$.  By  Proposition \ref{p:K-increaseness} e) we find that $g$ is $K_F$-increasing.

\smallskip

\noindent
ii)$\Rightarrow$iii): From Proposition \ref{p:K-increaseness} e).
\smallskip

\noindent
iii)$\Rightarrow$i):  Let $K:=K_F$. Clearly, $F$ is $K$-convex and, by Proposition \ref{p:K-increaseness} e),   $g$ is $K$-increasing.
\end{proof}

\noindent
Proposition \ref{p:condition-F-Kconvex-g-Kincreasing} yields the following concrete example.

\begin{example}
\label{e:counter-example-of-goF-convex=>g-K-inc-and-F-K-convex}
Consider $g\colon (x,y)\in\bR^2\mapsto |x|$ and $F\colon (x,y)\in\bR^2\mapsto (x^2,y)$. Hence $g\in\Gamma_0(\bE_2)$ and $g\circ F\colon (x,y)\mapsto x^2 \in\Gamma_0(\bE_1)$. Using Proposition \ref{prop:KF} b), we find that $K_F=\bR_+\times\{0\}$. However, $(-1,0)\leqslant_{K_F} (0,0)$ and $g((-1,0))=1 > 0 = g((0,0))$. Thus, $g$ is not $K_F$-increasing, and  consequently, by Proposition \ref{p:condition-F-Kconvex-g-Kincreasing}, there is no closed cone $K$ such that $F$ is $K$-convex and $g$ is $K$-increasing.
\end{example}

\noindent
We close out by remarking that, if  $\phi\colon\bE_1\to\bR\cup\{+\infty\}$ is proper convex, there always exists a decomposition $\phi = g\circ F$ with $g\in\Gamma_0(\bE_2)$, $F\colon\bE_1\to\bE_2^\bullet$ proper with $g$ $K_F$-increasing: for instance, define $F(x):=(\Phi(x),0,\dots,0)\in\bE_2$ with $\dom F=\dom \Phi$, and $g(y) =y_1$. Then, $g\in\Gamma_0(\bE_2)$, $F$ is $(\bR_+\times\{0\}\times\dots\times\{0\})$-convex and $g$ is $(\bR_+\times\{0\}\times\dots\times\{0\})$-increasing.

\subsection{Beyond \boldmath{$K$}-monotonicity} It was already observed  by Pennanen \cite{Pen 99} and Burke et al. \cite{BHN 21}  that,  in order to obtain convexity of the composition $g\circ F$ in \eqref{eq:CC},
the assumption that $g$ be $K$-increasing can be weakened to 
\begin{equation}\label{eq:KEpiMon}
g(F(x))\leq g(y)\quad \forall (x,y) \in \Epi{K}{F},
\end{equation}
in which case we call {\em  $g$ increasing w.r.t $\Epi{K}{F}$}. 
Concretely, the following result holds.

\begin{proposition}[{\cite[Proposition 1]{BHN 21}}] \label{prop:KEpiMon}
Let $K\subset\bE_2$ be  a convex cone such that that $F:\bE_1\to\bE_2^\bullet$ is $K$-convex and  such that  $g\in\Gamma(\bE_1)$ is   increasing w.r.t $\Epi{K}{F}$ in the sense of \eqref{eq:KEpiMon}.  Then $g\circ F$ is convex.
\end{proposition}

\noindent
The next proposition gives a characterization of  the situation where there exists a closed convex cone $K$ such that $g\in\Gamma_0(\bE_1)$ is  increasing w.r.t $\Epi{K}{F}$ and $F$ is $K$-convex.

 \begin{proposition}
Let $g\in\Gamma_0(\bE_2)$ and $F:D\to\bE_2$  for $D\subset\bE_1$ (nonempty convex) such that  $g\circ F$ is proper. Then there exists a closed (convex) cone $K$ such that $g$ is  increasing w.r.t $\Epi{K}{F}$ (in the sense of \eqref{eq:KEpiMon}) and  $F$ is $K$-convex if and only if $g$ is   increasing w.r.t $\Epi{K_F}{F}$.
\end{proposition}

\begin{proof}
Suppose that $g$ is   increasing w.r.t $\Epi{K_F}{F}$ and  set $K:=K_F$. Then  $K$ is closed and convex and $F$ is $K$-convex and $g$ is   increasing w.r.t $\Epi{K}{F}$ (by assumption).

Conversely, suppose now that there exists a closed (convex) cone $K$ such that $g$  is   increasing w.r.t $\Epi{K}{F}$ and $F$ is $K$-convex. As $K$ is closed and $F$ is  $K$-convex, we have$K_F\subset K$, by definition of  $K_F$. Therefore
$\Epi{K_F}{F} =\gph F+\{0\}\times K_F\subset \gph F+\{0\}\times K=\Epi{K_F}{F}$, and thus  $g$  is   increasing w.r.t $\Epi{K_F}{F}$.

\end{proof}

\noindent
It turns out that, in Proposition \ref{prop:KEpiMon}, the assumption that $g$ be increasing w.r.t. to $\Epi{K}{F}$ can even be further weakened  substituting  $\co(\gph F)$  for  $\Epi{K}{F}$ by  which, again, ties our considerations here to our previous study.

\begin{proposition}
Let $D\subset \bE_1$ be (nonempty) convex, $F:D\to\bE_2$, and let  $g\in\Gamma(\bE_2)$ be increasing w.r.t. $\co(\gph F)$, i.e.
\begin{equation}\label{eq:MinInc}
g(F(x))\leq g(y)\quad\forall (x,y)\in\co(\gph F).
\end{equation}
Then $g\circ F$ is convex.
\end{proposition}

\begin{proof}
Let $x,y\in\dom(g\circ F)$ and $\alpha\in(0,1)$. Then $(\alpha x+(1-\alpha)y,\alpha F(x)+(1-\alpha)F(y))\in\co(\gph F)$. By \eqref{eq:MinInc},and the convexity of $g$ we find  
\[
 g(F (\alpha x+(1-\alpha)y))\leq g(\alpha F(x)+(1-\alpha)F(y)) \leq \alpha g(F(x))+(1-\alpha)g(F(y)).
\]
Hence $g\circ F$ is convex.
\end{proof}

\subsection*{Acknowledgments} We would like to thank Prof. James V. Burke (University of Washington) for very useful discussions on the material presented here.


\begin{thebibliography}{99}


\bibitem{BaC 17}{\sc H.H. Bauschke and P.L. Combettes:}
{\em Convex analysis and Monotone Operator Theory in Hilbert Spaces.}
CMS Books in Mathematics, Springer, New York, 2nd Edition,  2017.

\bibitem{BoV 04}{\sc S.~Boyd and L.~Vandenbergh:}
{\em Convex Optimization.}
Cambridge University Press, 2004.

\bibitem{Bor 74}{\sc J. M. Borwein:}
{\em Optimization with respect to Partial Orderings.}
Ph.D. Thesis, University of Oxford, 1974.

\bibitem{Bot 10}{\sc R.I.~Bo\c{t}:}
{\em Conjugate Duality in Convex Optimization. Lecture Notes in Economics and Mathematical Systems, 637.} 
Springer, Berlin, 2010.

\bibitem{Bot et al. 06} {\sc R.I.~Bo\c{t}, I.B.~Hodrea, and G.~Wanka:}
{\em Farkas-type results for inequality systems with composed convex functions via conjugate duality.}
Journal of Mathematical Analysis and Applications 322(1), 2006, pp.~316--328.

\bibitem{BGW 07}{\sc R.I.~Bo\c{t}, S.-M.~Grad, and G.~Wanka:} {\em New constraint qualification and conjugate duality for composed convex optimization problems.} Journal of Optimization Theory and Applications 135(2), 2007, pp. 241--255.

\bibitem{BuH 15}
{\sc J.~V.~Burke and T.~Hoheisel:} {\em Matrix support functionals for inverse
  problems, regularization, and learning.} SIAM Journal on Optimization  25(2),
  2015, pp.~1135--1159.

\bibitem{BGH 17} {\sc J.~V.~Burke, Y.~Gao and T.~Hoheisel:}
{\em Convex geometry of the generalized matrix-fractional function.}
SIAM Journal on Optimization 28(3), 2018, pp.~2189--2200.


\bibitem{BGH 18} {\sc J.~V.~Burke, Y.~Gao, and T.~Hoheisel:} {\em  Variational properties of matrix functions via the generalized matrix-fractional function.}  SIAM Journal on Optimization  29(3), 2019, pp.~1958--1987.
%


\bibitem{BHN 21} {\sc J.V.~Burke, T.~Hoheisel, and Q.~Van~Nguyen:}{\em A study of convex convex-composite functions via infimal convolution with applications.}  Mathematics of Operations Research, 2021,  https://doi.org/10.1287/moor.2020.1099.



\bibitem{HiH 06} {\sc J.-B.~Hiriart-Urruty:}
{\em A Note on the Legendre-Fenchel Transform of Convex Composite Functions.}
In {\em Nonsmooth Mechanics and Analysis.} Eds.  P. Alart, O. Maisonneuve, and R. T. Rockafellar,
Springer, 2006, pp.~35--46.


\bibitem{HoJ 91}{\sc R.A.~Horn and C.R.~Johnson:} {\em Topics in Matrix Analysis.}
Cambridge University Press, New York, 1991. 



\bibitem{JFX 17} {\sc A.~Jalali, M.~Fazel, and L.~Xiao:}
{\em Variational Gram functions: convex analysis and optimization.}
SIAM Journal on Optimization 27(4), 2017, pp.~2634--2661. 

 \bibitem{KK95}{\sc A.G.~Kusraev and S.S.~Kutateladze:}
 {\em Subdifferentials: theory and applications.} 
 Mathematics and its Applications, 323. Kluwer Academic Publishers Group, Dordrecht, 1995.

\bibitem{Lew 96} {\sc A.S.~Lewis:}{\em Convex analysis on the Hermitian matrices.}
SIAM Journal on  Optimization  6(1), 1996, pp.~164--177. 


\bibitem{LLY 20} {\sc G.~Li , S.~Li, and  M.~You:}{\em Recession function and its
applications in optimization.}
Optimization, 2020,  https://doi.org/10.1080/02331934.2020.1786569.



\bibitem{Pen 99}
{\sc T.~Pennanen:}{\em Graph-convex mappings and $K$-convex functions.} 
Journal of Convex Analysis 6(2), 1999,  pp.~235--266.

\bibitem{Roc 70}
{\sc R.T. Rockafellar:} {\em Convex Analysis.} Princeton Mathematical Series, No. 28. Princeton University Press, Princeton, N.J. 1970.

\bibitem{RoW 98}
{\sc R.T.~Rockafellar and R.J.-B.~Wets:} {\em Variational Analysis.}
  Grundlehren der Mathematischen Wissenschaften, Vol.~317, Springer-Verlag, Berlin, 1998.
  


\end{thebibliography}
\end{document}